\documentclass[12pt]{amsart}
\usepackage[dvipsnames]{color}
\usepackage{amsfonts,amssymb,amsmath,amscd,amstext}
\usepackage[colorlinks=true,linkcolor=blue,citecolor=blue]{hyperref}
\usepackage[utf8]{inputenc}
\usepackage{graphicx}
\usepackage{changes}
\usepackage{comment}
\usepackage{mathdots}
\usepackage[a4paper,top=2.5cm,bottom=2.5cm,left=2cm,right=2cm]{geometry}
\renewcommand{\leq}{\leqslant}
\renewcommand{\geq}{\geqslant}
\renewcommand{\le}{\leqslant}
\renewcommand{\ge}{\geqslant}
\newcommand{\ptl}{\partial}

\newcommand{\norm}[1]{\| #1 \|}

\newcommand{\rr}{{\mathbb{R}}}

\newcommand{\hh}{{\mathbb{H}}}
\newcommand{\nn}{{\mathbb{N}}}

\newcommand{\la}{\lambda}

\newcommand{\Om}{\Omega}
\newcommand{\eps}{\varepsilon}

\newcommand{\ga}{\gamma}

\newcommand{\escpr}[1]{\langle#1\rangle}

\newcommand{\mh}{\mathcal{H}}

\newcommand{\df}{d_{K_,\partial\Om}}
\newcommand{\dfz}{d_{K_0,\partial\Om}}

\newcommand{\scu}{\longrightarrow}

\definechangesauthor[color=red]{J}
\definecolor{champagne}{rgb}{0.97, 0.91, 0.81}

\definecolor{asparagus}{rgb}{0.53, 0.66, 0.42}

\DeclareMathOperator{\divv}{div}
\DeclareMathOperator{\intt}{int}

\newtheorem{theorem}{Theorem}[section]
\newtheorem{proposition}[theorem]{Proposition}
\newtheorem{lemma}[theorem]{Lemma}

\theoremstyle{definition}

\newtheorem{remark}[theorem]{Remark}

\newtheorem{definition}[theorem]{Definition} 

\theoremstyle{remark}

\numberwithin{equation}{section}

\definechangesauthor[color=purple]{JP}
\definechangesauthor[color=blue]{S}
\definechangesauthor[color=blue]{G}
\begin{document}

\title[The PMC equation in the sub-Finsler Heisenberg Group]{The prescribed mean curvature equation for $t$-graphs in the Sub-Finsler Heisenberg group $\hh^n$}

\author[G. Giovannardi]{Gianmarco Giovannardi}
\address[Gianmarco Giovannardi]{Dipartimento di Matematica e Informatica "U. Dini", Università degli Studi di Firenze, Viale Morgani 67/A, 50134, Firenze, Italy}
\email{gianmarco.giovannardi@unifi.it}

\author[A.~Pinamonti]{Andrea Pinamonti}
\address[Andrea Pinamonti]{Dipartimento di Matematica\\ Università di Trento\\
Via Sommarive, 14, 38123 Povo TN}
\email{andrea.pinamonti@unitn.it}

\author[J.~Pozuelo]{Juli\'an Pozuelo} 
\address[Juli\'an Pozuelo]{Dipartamento di Matematica “Tullio Levi-Civita” \\
Università di Padova \\ Via Trieste, 63, 35131, Padova, Italy}
\email{julian.pozuelodominguez@unipd.it}

\author[S.~Verzellesi]{Simone Verzellesi} 
\address[Simone Verzellesi]{Dipartimento di Matematica\\ Università di Trento\\
Via Sommarive, 14, 38123 Povo TN}
\email{simone.verzellesi@unitn.it}

\date{\today}
\thanks{
The first, the second and the fourth-named authors are members of the Istituto Nazionale di Alta Matematica (INdAM), Gruppo Nazionale per l'Analisi Matematica, la Probabilità e le loro Applicazioni (GNAMPA).
 G. Giovannardi is supported by INdAM–GNAMPA 2022 Project \emph{Analisi geometrica in strutture subriemanniane} and by MIUR-PRIN 2022 Project \emph{Geometric-Analytic Methods
for PDEs and Applications}. 
A. Pinamonti and G. Giovannardi are supported by MIUR-PRIN 2017 Project \emph{Gradient flows, Optimal Transport and Metric Measure Structures}.  J. Pozuelo is supported by the research grant PID2020-118180GB-I00 funded by MCIN/AEI/10.13039/501100011033 and  Fundación Ramón Areces grants \emph{XXXV convocatoria para la amplación de estudios en el extranjero en ciencias de la vida y la materia}.
S. Verzellesi is supported by INdAM–GNAMPA 2023 Project \emph{Equazioni differenziali alle derivate parziali di tipo misto o dipendenti da campi di vettori}. A. Pinamonti and S. Verzellesi are supported by MIUR-PRIN 2022 Project \emph{Regularity problems in sub-Riemannian structures}.}

\bibliographystyle{abbrv} 

\maketitle

\thispagestyle{empty}
\begin{abstract}
    We study the sub-Finsler prescribed mean curvature equation, associated to a strictly convex body $K_0 \subseteq \rr^{2n}$, for $t$-graphs on a bounded domain $\Omega$ in the Heisenberg group $\hh^n$. When the prescribed datum $H$ is constant and strictly smaller than the Finsler mean curvature of $\partial \Omega$, we prove the existence of a Lipschitz solution to the Dirichlet problem for the sub-Finsler CMC equation by means of a Finsler approximation scheme. 
\end{abstract}

\section{Introduction}
The aim of this work is to study the prescribed mean curvature equation for $t$-graphs in the  Heisenberg group $\hh^{n}$ with a sub-Finsler structure. In the Heisenberg group, which can be identified with $\rr^{2n+1}$ endowed with a suitable non-Euclidean group law, a sub-Finsler structure is defined by means of an asymmetric left-invariant norm $\norm{\cdot}_{K_0}$ on the horizontal distribution of $\hh^n$ associated to a convex body $K_0\subseteq\rr^{2n}$ containing the origin in its interior. Let $\Omega \subseteq \rr^{2n}$ be a bounded open set, $H \in L^{\infty}(\Om)$, $F\in L^{1}(\Om,\rr^{2n})$ and $u \in W^{1,1}(\Om)$. We consider the functional
\begin{equation}
\label{int:I}
     \mathcal{I}(u)=\int_{\Omega} \|\nabla u +F\|_{K_0,*} \, dx dy + \int_{\Omega} H u \, dx dy,
\end{equation}
where $\| \cdot \|_{K_0,*}$ denotes the dual norm of $\|\cdot\|_{K_0}$. In particular, when $F(x,y)=(-y,x)$ the first term in \eqref{int:I} coincides with the sub-Finsler area of the $t$-graph of $u$, see \cite{PozueloRitore2021,2020arXiv200711384F}. Moreover, if $K_0$ is the Euclidean unit ball centered at the origin and $H=0$ then \eqref{int:I} boils down to the classical area functional for $t$-graphs in Heisenberg group, see \cite{MR2165405,MR2609016} and references therein.
We say that the graph of $u$ has prescribed $K_0$-mean curvature $H$ in $\Om$ if $u$ is a minimizer of $\mathcal{I}$. Indeed, the Euler-Lagrange equation associated to $\mathcal{I}$ out of the singular set $\Om_0$, i.e. the set of points where $\nabla u + F$ vanishes, is given by 
\begin{equation}
\label{int:pmce}
\divv( \pi_{K_0} (\nabla u + F))=H,
\end{equation}
where $\pi_{K_0}$ is a suitable $0$-homogeneous function defined in \eqref{sup2}. 
When we fix a boundary datum $\varphi\in W^{1,1}(\Om)$, a solution to the \textit{Dirichlet problem} for the prescribed $K_0$-mean curvature equation is a minimizer $u$ of $\mathcal{I}$ such that $u - \varphi$ belongs to the Sobolev space $W^{1,1}_0(\Omega)$.
Our main result is Theorem \ref{th:main}, where we prove, under suitable regularity assumptions on the data, that there exists a Lipschitz solution to the Dirichlet problem for the prescribed $K_0$-mean curvature equation when $H$ is \textit{constant}, it satisfies 
\begin{equation}
\label{eq:condHcost}
   |H|<H_{K_0, \partial \Om}(z_0)
\end{equation}for each $z_0=(x_0,y_0) \in \partial \Om$ and 
\begin{equation}
\label{eq:hyHnc}
    \left|\int_\Om H v\,dxdy\right| \leq (1-\delta)\int_{\Om}\|\nabla v\|_{K_0,*}\,dxdy
\end{equation}
for each non-negative function $v \in C_c^{\infty}(\Om)$ and a suitable $\delta=\delta(K_0,\Om,H)\in(0,1]$.
Here $H_{K_0, \partial \Om}$ denotes the Finsler mean curvature of the boundary $\partial \Om \subseteq \rr^{2n}$. Notice that the mean curvature of the graph of $u$ is computed with respect to the downward pointing unit normal and the Finsler mean curvature of $\partial \Om$ is computed with respect to the inner unit normal. The upper bound \eqref{eq:condHcost} of $H$ in terms of the Finsler mean curvature of the boundary is the Finsler analogous of the standard assumption for the solution to the Dirichlet problem for the classical mean curvature equation in the Euclidean setting as stated in \cite{MR282058}, \cite{MR336532} or \cite{MR0377669} (see also \cite[Theorem 16.11]{GT}).
On the other hand, \eqref{eq:hyHnc} is a standard sufficient condition for the estimates of the supremum of $|u|$ (see \cite{MR0377669} or \cite{GT}). It is worth noting that, in the Euclidean setting (cf. e.g. \cite{MR487722}), the weaker condition 
\begin{equation}
\label{eq:hyHncweak}
    \left|\int_\Om H v\,dxdy\right| \leq\int_{\Om}\|\nabla v\|_{K_0,*}\,dxdy
\end{equation}
for each $v \in C_c^{\infty}(\Om)$ is actually a necessary condition for the existence of a solution to the Euclidean prescribed mean curvature equation. Moreover, the Euclidean version of \eqref{eq:hyHncweak} suffices to guarantee existence of solutions to the Euclidean prescribed mean curvature equation as long as no boundary conditions are imposed (cf. \cite{MR487722,MR3767675}). 

Remarkably, as we will show in Section \ref{sc:aprioriestimate} and \ref{sec:last}, there are particular settings in which Theorem \ref{th:main} continues to hold even without imposing \eqref{eq:hyHnc}, such as the \emph{first sub-Finsler} Heisenberg group $\mathbb H^1$ (cf. Theorem \ref{th:main23}) and \emph{any sub-Riemannian} Heisenberg group $\hh^n$ (cf. Theorem \ref{th:main2}).
The Dirichlet problem for constant mean curvature in the first Riemannian Heisenberg group has been studied in \cite{MR2332426} under the same condition on the mean curvature.  It is worth mentioning that this is the first time that the existence of Lipschitz solutions to the sub-Finsler Dirichlet problem has been studied when $H\ne0$, even in the particular case in which $K_0$ is the unit disk centered at $0$, where the sub-Finsler and the sub-Riemannian frameworks coincide. Indeed, as far as we know, the sub-Riemannian Dirichlet problem has been studied in \cite{MR2043961,MR2262784,MR2165405, MR3216825, DLPT, PSCTV} only in the case of minimal surfaces under the bounded slope condition or the $p$-convexity assumption on $\Om$, and in \cite{MR4497298} when $H\neq 0$ is small enough and in a weaker functional framework. In particular, we point out that when $n=1$ our assumption \eqref{eq:condHcost} implies that $\Omega \subseteq \rr^2$ is strictly convex, see Remark \ref{rk:convex}. It is easy to check that our sub-Finsler functional $\mathcal{I}$ for $H=0$ satisfies the hypothesis of the area functional considered in \cite{DLPT}. Thus, assuming the bounded slope condition we directly obtain the existence of Euclidean Lipschitz minimizers for Plateau's problem. The approach of the present paper, based on the Schauder fixed-point theory, follows the scheme developed in \cite{MR2262784} and extends its results both to the case of prescribed constant mean curvature  $H \ne0$ and to the sub-Finsler setting. 
In Theorem \ref{th:main} we cannot expect better regularity than Lipschitz. Indeed, even in the sub-Riemannian Heisenberg group $\hh^1$ there are
several examples of non-smooth area minimizers. For instance, S.D. Pauls \cite{MR2225631} exhibited a solution of low
regularity for Plateau's problem with smooth boundary datum,
while in \cite{MR2262784,MR2448649,SGR}
the authors provided solutions to the Bernstein problem in $\hh^1$ that are only Euclidean Lipschitz. These examples have been recently generalized to the sub-Finsler setting in \cite{GPR22}. We refer the interested reader to \cite{GR24} for a positive result to the sub-Finsler Bernstein problem for $(X,Y)$-Lipschitz surfaces, which can be seen as a regularity result for global perimeter minimizers.

Since equation \eqref{int:pmce} is sub-elliptic degenerate and it is singular next to the singular set, inspired by \cite{MR2262784, MR2043961}, we first introduce a family of desingularized approximating equations given by 
\begin{equation}\label{asx}
	\divv\left(\pi_{K_0}(\nabla u+F)\frac{\|\nabla u+F\|^2_{\ast}}{\left(\varepsilon^3+\|\nabla u+F\|_{\ast}^3\right)^\frac{2}{3}}\right)=H
\end{equation}
for each $0<\eps<1$. A  similar approximation scheme was considered in the sub-Riemannian setting in \cite{MR2774306,MR2583494} to study the Lipschitz regularity for non-characteristic minimal surfaces. For a detailed analysis of this approach, we refer to \cite{MR3510691}. This family of equations can be obtained by considering a $(2n+1)$-dimensional convex body $K_{\eps}$ containing the origin in its interior, that converges in the Hausdorff sense to the $2n$-dimensional convex body  $K_0$ as $\eps\to 0$. The choice of the convex body $K_{\eps}$ is not arbitrary. Indeed, we need a specific shape in order to obtain an approximating equation well-defined in the classical sense in the singular set. It is interesting to point out that the  Riemannian approximation of \cite{MR2262784, MR2043961,MR2774306,MR2583494} produces an approximation of the unit disk $D\subseteq \rr^{2n}$ by ellipsoids in the sub-Riemannian setting, and this approximation does not work in the greater sub-Finsler generality. 
Indeed, if instead of \eqref{asx} we were to consider the more natural equation
\begin{equation}\label{riemapprox62}
    \divv\left(\pi_{K_0}(\nabla u+F)\frac{\|\nabla u+F\|_*}{\sqrt{\eps^2+\|\nabla u+F\|^2_*}}\right)=H,
\end{equation}
reminiscent of the Riemannian approximation scheme of \cite{MR2262784} (cf. Remark \ref{reminfondo}), we would have to require certain assumptions on $K_0$ for \eqref{riemapprox62} to be well-defined in the classical sense. We refer to Section \ref{sec:last} for a more careful analysis in this regard. On the other hand, while \eqref{asx} is always well-defined, it still tends to degenerate close to the singular set, so that it could fail to be elliptic. Therefore, we need to regularize \eqref{asx} by perturbing it with an Euclidean curvature term. More precisely, we consider the family of equations given by
\begin{equation}\label{findiv24intro}
	\divv\left(\pi(\nabla u+F)\frac{\|\nabla u+F\|^2_{\ast}}{\left(\varepsilon^3+\|\nabla u+F\|_{\ast}^3\right)^\frac{2}{3}}\right)+\eta\divv\left(\frac{\nabla u+F}{\sqrt{1+|\nabla u+F|^2}}\right)=H.
\end{equation}
for any $\eps\in (0,1)$ and any $\eta>0$ sufficiently small, whose associated Finsler variational functional is given by 
\[
\mathcal{I}_{\eps ,\eta}(u)=\int_{\Omega} \left(\eps^3+ \|\nabla u +F\|_{K_0,*}^3\right)^{\frac{1}{3}} \, dx dy + \eta\int_\Om\sqrt{1+|\nabla u+F|^2}\, dx dy + \int_{\Omega} H u \, dx dy.
\]
A direct computation (cf. Section \ref{sc:aprioriestimate}) will show that \eqref{findiv24intro} is in fact a classical, quasi-linear second-order elliptic equation. 
Therefore, given a boundary datum $\varphi \in C^{2,\alpha}(\bar{\Om})$, the solvability of the Dirichlet problem associated to \eqref{asx} is reduced by \cite[Theorem 13.8]{GT} to \emph{a priori} estimates in $C^1(\overline{\Om})$ of a related family of problems. As usual the \emph{a priori} estimates in $C^1(\overline{\Om})$ consist of three parts: estimates of the supremum of $|u|$, boundary estimates of the gradient of $u$ and interior estimates of the gradient of $u$. While the estimates of the supremum rely on assumption \eqref{eq:hyHnc}, the boundary estimates of the gradient are obtained by a barrier argument that depends on the Finsler distance from the boundary $\partial \Om$. 
Due to technical reasons in the construction of the barriers we need to assume the strict inequality in  \eqref{eq:condHcost}, avoiding the optimal case when $H$ coincides with $H_{K_0, \partial \Om}(z_0)$ at a given point $z_0 \in \partial\Om$.
We emphasize that these results hold even if the prescribed curvature $H$ is non-constant and Lipschitz. The only crucial step where we need $H$ to be constant is the maximum principle for the gradient of the solution that allows us to reduce the interior estimates of the gradient to its boundary estimates. Finally, once we realize that $C^1$ estimates are independent of the approximation parameters $\eps$ and $\eta$, passing to the limit as $\eps,\eta\to 0$ and using Arzelà-Ascoli Theorem we get the existence of a Lispchitz minimizer for the sub-Finsler Dirichlet problem. 

In the last decades, variational problems related to the sub-Riemannian area introduced by Capogna, Danielli and Garofalo \cite{MR1312686}, Garofalo and Nhieu \cite{MR1404326} and Franchi, Serapioni and Serra Cassano \cite{MR1871966} have received great interest, see also \cite{MR1404326, MR2354992, MR2262784,MR2165405, MR2312336, MR2472175, MR2333095,  MR2223801, MR2401420,MR3048517,MR2609016,MR2435652, MR2358000, MR3412408, MR2583494, MR2235475,MR2177813,MR4316814,MR3794892,MR2165405}. The monograph \cite{MR2312336} provides a quite complete survey of progresses on the subject.

In particular, the analysis of the Dirichlet problem with $H\ne 0$ constant for $t$-graphs is essential since it is strictly related to the isoperimetric problem in $\hh^n$. In \cite{MR676380}, P. Pansu conjectured that the boundaries of isoperimetric sets in $\hh^1$ are given by the surfaces now called Pansu's spheres, obtained as the union of all sub-Riemannian geodesics of a fixed curvature joining two points in the same vertical line. This conjecture has been solved only assuming \emph{a priori} some regularity of the minimizers of the area with constant prescribed mean curvature, such as the $C^2$ regularity of the minimizers \cite{MR2435652}, the axial symmetry of the minimizers \cite{MR2402213}, the Euclidean convexity of minimizers \cite{MR2548252} and when the isoperimetric set both contains a horizontal disk $D_r$ and is contained in a vertical cylinder $C_r$ for some $r>0$ (cf. \cite{MR2898770} for a more accurate statement). Recently, in \cite{PozueloRitore2021,2020arXiv200711384F}, the notion of Pansu's spheres has been generalized to the Pansu-Wulff spheres in the sub-Finsler setting. We refer to \cite{AS17} for earlier research in this direction. Consequently, the results presented in \cite{MR2435652} and \cite{MR2898770} have been generalized in \cite{2020arXiv200711384F} and \cite{PozueloRitore2021} respectively. Finally, in \cite{MR4314055} the $C^2$ regularity of the characteristics curves for the prescribed $K_0$-mean curvature equation with continuous datum $H$ is established when the boundary of the set is  Euclidean Lispchitz and $\hh$-regular. Hence our existence result in the class of Lispchitz $t$-graphs provides an important contribution to the understanding of the sub-Finsler isoperimetric problem. 
Recently, similar results concerning CMC graphs and surfaces in the Euclidean setting with an anisotropic norm have been obtained by \cite{DRT22,DPDR22}. \\

The manuscript is organized as follows. In Section \ref{sc:preliminaries} we introduce some preliminary definitions and results, such as the Minkowski norm, the Finsler geometry of a hypersurface in $\rr^{2n}$, the Heisenberg group, the sub-Finsler perimeter and the sub-Finsler functional $\mathcal{I}$. Section \ref{sc:Fap} is dedicated to the Finsler approximation by the $K_{\eps}$ convex body of the sub-Finsler convex body $K_0$. Section \ref{sc:aprioriestimate} deals with the \emph{a priori} estimates for the $C^1$ norm of the solution to the approximating elliptic equations \eqref{findiv24intro}. In particular, Proposition \ref{th:Cinfestimate} deals with the \emph{a priori} estimates of $|u|$ when $H$ is Lispchitz and verifies the integral condition \eqref{eq:hyHnc}, in Proposition \ref{prop:boundgradest} we deduce the boundary estimates of the gradient when $H$ is Lispchitz, in Proposition \ref{prop:maxprinc} we establish the maximum principle for the gradient for $H$ constant, and finally, in Proposition \ref{th:Cinfestimateconst} we achieve \emph{a priori} estimates of $|u|$ when $H$ is constant and $n=1$. Section \ref{sc:mainth} contains the main Theorems \ref{th:main} and \ref{th:main23}. Finally, in Section \ref{sec:last} we prove Theorem \ref{th:main2} in the sub-Riemannian setting by means of the approximation scheme associated to \eqref{riemapprox62}.

\subsubsection*{Acknowledgement}
The authors warmly thank Manuel Ritoré for his advice and for stimulating discussions and YanYan Li for fruitful conversations about the content of the present paper. The authors would also like to thank the anonymous referees for giving such constructive comments which substantially helped improve the quality of the paper.

 \section{Preliminaries}
 \label{sc:preliminaries}
\subsection{Notation}\label{notsec} Unless otherwise specified, we let $n,d\in\mathbb{N}$, $n,d\geq 1$. Given two open sets $A,B\subseteq \rr^d$, we write $A\Subset B$ whenever $\overline A\subseteq B$. We say that a set $K$ is a convex body if it is convex, compact and has non-empty interior. We say that a convex body $K$ is (in) $C^{k,\alpha}_+$, for $k\in \nn$ and $\alpha\in [0,1]$, if $\partial K$ is of class $C^{k,\alpha}$ with  strictly positive principal curvatures.
 \subsection{Minkowski norms}
 We follow the approach developed in \cite{PozueloRitore2021, MR3155183}.
We say that $\|\cdot\|:\rr^d\to [0,+\infty)$ is a \emph{norm} if it verifies:
\begin{enumerate}
	\item $\|v\|=0\Leftrightarrow v=0$,
	\item $\|sv\|=s\|v\|$ for any $s>0$,
	\item $\|v+u\|\leq\|v\|+\|u\|$
\end{enumerate}
for any $u,v\in\rr^d$. We stress the fact that we are not assuming the symmetry property $\norm{-v}=\norm{v}$. It is well known that any norm is equivalent to the Euclidean  norm $|\cdot|$, that is, given a norm $\|\cdot\|$ in $\rr^d$ there exist constants $0<c<C$ such that
\begin{equation}\label{in:norms}
    c |\cdot |\leq\|\cdot\|\leq C |\cdot|.
\end{equation}
Associated to a given a norm $\|\cdot\|$ we have the set $F=\{u\in \rr^d:||u||\le 1\}$, which,  thanks to \eqref{in:norms} and the properties of $\|\cdot\|$, is compact, convex and includes $0$ in its interior. Reciprocally, given a convex body $K$ with $0\in\intt(K)$, the function 
$$||u||_K=\inf\{\la\ge 0:  u\in\la K\}$$
 defines a norm so that $K=\{u\in\rr^d:||u||_K\le 1\}$. In the following we let 
 $$B_{K}(v,r):=\{w\in \rr^d\,:\,\|w-v\|_K\leq r\}$$
for any $v\in \rr^d$ and $r>0$.
It is easy to check that 
$\|v\|_K=\|-v\|_{-K}$ for any $v\in\rr^d$, so that 
\begin{equation}\label{pallastorta}
    B_{-K}(v,r):=\{w\in \rr^d\,:\,\|v-w\|_K\leq r\}
\end{equation}
for any $v\in\rr^d$ and $r>0$.
Given a norm $\norm{\cdot}$ and a scalar product $\escpr{\cdot,\cdot}$ in $\rr^d$, we consider the dual norm $\norm{\cdot}_*$ of $\|\cdot\|$ with respect to $\escpr{\cdot,\cdot}$, defined by
\begin{align}\label{sup}
\norm{u}_*=\sup_{\norm{v}\le 1}\escpr{u,v}.
\end{align}
The dual norm is the support function of the unit ball $F$ with respect to the scalar product $\langle \cdot,\cdot\rangle$. Moreover, thanks to the above definitions the following Cauchy-Schwarz formula holds:
\begin{equation}\label{in:C-S}
    \escpr{u,v}\leq \|u\|_*\|v\|
\end{equation}
for any $u,v\in\rr^d$.
If in addition we assume $K$ to be strictly convex and $u\neq 0$, then the compactness and strict convexity of $K$ guarantee the existence of a unique vector $\pi_K(u)$ in $\partial K$ where the supremum in \eqref{sup} is attained, i.e.
\begin{equation}\label{sup2}
\|u\|_{K,*}=\langle u,\pi_K(u)\rangle.
\end{equation}
It is easy to see that $\pi_K$ is a positively $0$-homogeneous map, i.e. $\pi_K(\lambda u)=\pi_K(u)$ for any $\lambda>0$ and $u\in \mathbb{R}^d\setminus\{0\}$, and that $\|\pi_K(u)\|_K=1$ for any $u\in\rr^d\setminus\{0\}$. Moreover, if we assume that $K$ is $C^2_+$,  then $\pi_K|_{\mathbb{S}^{d-1}}:\mathbb{S}^{d-1}\scu\partial K$ is a $C^1$ diffeomorphism whose inverse is the Gauss map $\mathcal{N}_K$ of $\partial K$ with respect to the outer unit normal. In particular,
\begin{equation}\label{positivedefinite}
    D\pi_K(u)\text{ is positive definite}
\end{equation}
for any $u\in\rr^{d}\setminus\{0\}$.
Furthermore, we have that the norms $\| \cdot \|_K$ and $\| \cdot \|_{K,*}$ belong to $C^{k,\alpha}(\rr^{d}\setminus \{0\})$ if and only if $\partial K$ is $C^{k,\alpha}$ for $k \in \nn$ and $0 \le \alpha \le1$. For further details see \cite[Section 2.5]{MR3155183}. The relation between the dual norm and the map $\pi_{K}$ is given by
\begin{equation}\label{eq:normpi}
\nabla \|u\|_{K,\ast}=\pi_{K}(u).
\end{equation}
Indeed, for any $u\in\rr^d\setminus\{0\}$
\begin{equation*}
    \nabla \|u\|_{K,\ast}=\nabla\langle u,\pi_{K}(u)\rangle=\pi_{K}(u)+u\cdot D\pi_K(u)=\pi_K(u),
\end{equation*}
where the last equality follows from the fact that $0$-homogeneous functions are radial.

\subsection{Finsler geometry of hypersurfaces in the Euclidean space}\label{subjulian}
Let $K\subseteq\rr^d$ be a convex body in $C^2_+$, $0\in \intt K$ and  $\Om\subseteq \rr^d$ be a bounded domain with boundary $\partial\Om=\Sigma$ of class $C^2$. Let $N$ be the inner unit normal to $\Sigma$. Then the derivative map $(W_{K,\Sigma})_p=-d_p(\pi_{K}\circ N): T_p\Sigma \to T_{\pi_K(N(p))}\partial K$, being $\pi_K$ as in \eqref{sup2}, is called the \emph{$K$-Weingarten map}. Let $\gamma\subseteq \partial K$ be a differentiable curve with $\gamma(0)=\pi_{K}(N(p))$ and $\gamma'(0)\in  T_{\pi_{K}(N(p))}\partial K$. By definition of $\pi_{K}$, the function
\[
f(t)=\escpr{\gamma(t),N(p)}
\]
has a maximum at $0$ and therefore $\langle \ga'(0),N(p)\rangle =f'(0)=0$, which gives $T_{\pi_{K}(N(p))}\partial K=T_{N(p)} \mathbb{S}^{d-1}$. Moreover it is well known that $(dN)_q$
is an endomorphism of $T_q\Sigma$ and therefore $(W_{K,\Sigma})_p$ is an endomorphism of $T_p\Sigma$. We define the \emph{$K$-mean curvature} of $\Sigma$ as
\[
H_{K,\Sigma}=\text{Trace}(W_{K,\Sigma})=-\divv _\Sigma(\pi_{K}\circ N),
\]
where $\divv_\Sigma$ is the divergence in the tangent directions to $\Sigma$. 
We remark that $W_{K,\Sigma}$ is neither necessarily self-adjoint nor symmetric. Let us check that $W_{K,\Sigma}$ is anyway diagonalizable. Indeed, given a parametrization $X$ of $\Sigma$, $dN$ has a symmetric matrix representation $S$ in the basis $B=\{\partial_{x_1}X,\ldots,\partial_{x_{d-1}}X\}$. On the other hand, $\pi_K=\mathcal{N}_K^{-1}$  and, since $K$ is in $C^2_+$,  the matrix $A$ which represents $d(\mathcal{N}_K^{-1})$ with respect to $B$ is positive definite. Therefore, there exists an invertible matrix $P$ such that $A=P^tP$. Notice that the matrices $P^tP S$ and $PSP^{t}$ have the same spectrum, and equal to the spectrum of $W_{K,\Sigma}$. Since $S$ is symmetric we can apply Sylvester's criterion
to obtain that all the eigenvalues of $PSP^t$ are real. The eigenvalues of $W_{K,\Sigma}$ are called \emph{$K$-principal curvatures} and the eigenvectors of $W_{K,\Sigma}$ are called \emph{$K$-principal directions}.

 \subsubsection{Finsler distance from the boundary and the Eikonal equation}\label{Julian}
 In this and the following section we want to rely on some results by \cite{MR2305073,MR2094267}, and so we assume that $K$ is in $ C^{\infty}_+$, i.e. $\partial K$ is of class $C^{\infty}$ with strictly positive principal curvatures.  Let $\Omega\subseteq\rr^d$ be a bounded domain with boundary $\partial\Omega=\Sigma$ of class $C^{2,\alpha}$, for $0<\alpha\leq 1$, and inner unit normal $N$. We shall adapt Theorem 4.26  in \cite{MR2522595} and the remarks at the end of Section 4.5 in \cite{MR2522595} to prove existence of a tubular neighborhood of $\Sigma$ and compute the $K$-mean curvature of parallel hypersurfaces. The \emph{interior signed $K$-distance}  to $\Sigma$ is the function $d_{K,\Sigma}:\rr^d\to\rr$ given by
 \begin{equation*}
 d_{K,\Sigma}(p)=\begin{cases}\min \{ \|p-q\|_K : q\in\Sigma \}&\text{ if  }p\in\Om\\
		-\min \{ \|p-q\|_K : q\in\Sigma \} &\text{ if } p\notin\Om.
	\end{cases}
 \end{equation*}
Consider the map $F:\Sigma\times\rr\to \rr^d$ given by
\[
F(q,t)=q+t(\pi_K\circ N)(q).
\]
For any $v\in T_q\Sigma$, we have $(dF)_{(q,t)}(v,0)=v+td(\pi_K\circ N)(v)$ and $(dF)_{(q,t)}(0,1)=(\pi_K\circ N)(q)$. Since $K$ contains the origin, 
\[
\escpr{\pi_K(N),N}>0
\]
and $dF$ is invertible at $t=0$. Thus $F$ is locally a diffeomorphism and, being $\Sigma$ a compact hypersurface, $F$ is a diffeomorphism in a domain $\Sigma\times(-\delta,\delta)$. The set $F(\Sigma\times(-\delta,\delta))$ is called a \emph{tubular neighborhood} of $\Sigma$. Notice that if $p=F(q,t)$, then
\begin{equation}\label{eq:precious}
p-q=t(\pi_K\circ N)(q)
\end{equation}
and, taking the $K$-norm, we obtain that $d_{K,\Sigma}(p)=t$. We know (cf. \cite{MR2305073}) that, under our assumptions, there exists $\bar\delta>0$ such that
\begin{equation*}
    d_{K,\Sigma}\in C^{2,\alpha}(\overline{F(\Sigma\times(-\delta,\delta))}).
\end{equation*}
for any $\delta<\bar\delta$.
Given $|t|<\delta$, we let
\begin{equation}\label{def:par}
\Sigma_t=\{p\in\rr^d : p=F(q,t) \text{ for some }q\in\partial \Sigma
\}.
\end{equation}
\begin{proposition}
Let $\Om\subseteq\rr^d$ be a bounded domain with boundary $\partial\Om=\Sigma$ of class $C^2$  and let $F(\Sigma\times(-\delta,\delta))$ be a tubular neighborhood of $\Sigma$. The $K$-mean curvature of $\Sigma_t$ at $p\in\Sigma_t$ is given by
\begin{equation}\label{eq:meanpar}
H_{K,\Sigma_t}(p)=\sum_{i=1}^{d-1}\frac{\kappa_i(q)}{1-t\kappa_i(q)},
\end{equation}
where $q\in\Sigma$ satisfies $p=F(q,t)$ and $\kappa_1(q),\ldots,\kappa_{d-1}(q)$ are the $K$-principal curvatures of $\Sigma$ at $q$.
\end{proposition}
\begin{proof}
Let $\{e_1,\ldots,e_{d-1}\}$ be a basis of $K$-principal directions of $\Sigma$. Then $(dF)_{(q,t)}(e_i,0)=(1-t\kappa_i)e_i$. Therefore a basis of principal directions in $\Sigma_t$ is $\{\frac{e_1}{1-t\kappa_1},\ldots,\frac{e_{d-1}}{1-t\kappa_{d-1}} \}$. Since we have  $$-d(\pi_K\circ N)_q \left(\frac{e_i}{1-t\kappa_i}\right)=\frac{\kappa_i}{1-t\kappa_i} e_i$$
for each $i=1,\ldots,d-1$ we get the conclusion.
\end{proof}
\begin{remark}
From \eqref{eq:meanpar}, we obtain that the $K$-mean curvature is increasing in $t$. In particular, given $q\in \Sigma$ and $p=F(q,t)$ for $t>0$, it holds that
\begin{equation}\label{in:compar}
H_{K,\Sigma_t}(p)\geq H_{K,\Sigma}(q).
\end{equation}
\end{remark}

 The following Eikonal equation can be deduced using classical arguments. We include the proof for the sake of completeness.
 \begin{proposition}\label{prop:findist}
 It holds that
 \begin{equation}\label{eq:Ei}
 \|\nabla d_{K,\Sigma}(p)\|_{K,*}=1
 \end{equation}
 for any $p$ where $\df$ is differentiable.
 \end{proposition}
 \begin{proof}
 It is clear that, for any $p,p'$ in $\rr^d$, we have
 \[
 d_{K,\Sigma}(p')\leq \|p'-p\|_K+d_{K,\Sigma}(p).
 \]
 Taking $p'=p+tv$ where $t>0$, we get
 \[
 d_{K,\Sigma}(p+tv)-d_{K,\Sigma}(p)\leq \|tv\|_K.
 \]
 Therefore, 
 \begin{equation}\label{eq:pre-e}
 \escpr{v,\nabla  d_{K,\Sigma}(p)}\leq \|v\|_K.
 \end{equation}
Taking $v=\pi_K(\nabla  d_{K,\Sigma}(p))$ in \eqref{eq:pre-e}, we obtain
\[
\|\nabla  d_{K,\Sigma}(p)\|_{K,*}\leq 1.
\]
On the other hand, let $\gamma(t)=F(q_0,t)$. By \eqref{eq:precious} we have that 
\[
d_{K,\Sigma}(\gamma(t))=t.
\]
Taking derivatives in the previous equation, we obtain
\[
\escpr{\gamma'(t),\nabla d_{K,\Sigma}(\gamma(t))}=1.
\]
Since $\gamma'(t)=(\pi_K\circ N)(q_0)$, we get that $\|\gamma'(t)\|_K=1$. Using \eqref{in:C-S}, we get
\[
\| \nabla d_{K,\Sigma}(\gamma(t))\|_{K,*}\geq 1. \qedhere
\]
\end{proof}

 Given a tubular neighborhood $\mathcal{O}$ of $\partial\Om$ and $p=F(q,t)\in\Om$, we denote by $N_t(p)$ the inner unit normal to $\Sigma_t$ at $p$. 
Let us explicitly compute $\divv(\pi_K\circ N_t)(p)$. Let us recall that, to the $0$-homogeneity of $\pi_K$, we get that
\[
q\cdot D\pi_K(q)=0
\]
for any $q\in\rr^d$. In particular, taking $q=N_t$, we obtain
\[
N_t\cdot D(\pi_K \circ N_t)=N_t\cdot D\pi_K(N_t)\cdot D N_t=0,
\]
which implies that
\begin{equation}\label{tange}
-\divv(\pi_K\circ N_t)(p)=-\divv _\Sigma(\pi_K\circ N_t)(p)=H_{K,\Sigma_t}(p)\geq H_{K,\partial\Om}(q).
\end{equation}

With the next result, we better understand the relationship between the Finsler mean curvature of $\Sigma$, the Euclidean curvature of $\Sigma$ and the Euclidean principal curvatures of $K$.
\begin{proposition}
\label{prop:FMCEwrtE}
Let $K$ be a convex body in $C^2_+$, $0\in \intt K$. Let $\Omega\subseteq \rr^d$ be a bounded domain with $\partial \Om=\Sigma$ of class $C^2$ and let $N_q$ be the inner unit normal to $\Sigma$ at $q$. Then we have 
\begin{equation}
\label{eq:FMCEwrtE}
 H_{K,\Sigma}(q)=-\sum_{i=1}^{d-1} \dfrac{\escpr{D_{e_i} N_q, e_i}}{k_i^{K}(\pi_K(N_q))}
\end{equation}
where $k_i^{K}$ are the Euclidean principal curvatures of $\partial K$ and $e_1,\ldots,e_{d-1}$ is an orthonormal basis  of Euclidean principal directions of $\partial K$.   

\begin{proof}
We shall drop the subscript for $\pi_K$. Let $q$ in $\Sigma$ and  $e_1,\ldots,e_{d-1}$ be an orthonormal basis of $\rr^{d-1}=T_{\pi(N_q)} \partial K$ such that 
\[
(d \mathcal{N}_K)_{\pi (N_q)} e_i=k_i^{K}(\pi(N_q)) e_i.
\]
By hypothesis, $k_i^K>0$
for $i=1,\ldots,d-1$.  Here $\mathcal{N}_K$ denotes the Gauss map of $\partial K$. Then we have 
\begin{align*}
H_{K,\Sigma}(q)=-\divv_{\Sigma} (\pi(N_q))=-\sum_{i=1}^{d-1} \escpr{D_{e_i} \pi (N_q),e_i},
\end{align*}
where $D$ is the Levi-Civita connection in $\rr^d$. We claim that $D_{e_i} \pi (N_q)= d\pi(D_{e_i} N_q)$. Indeed, let $\gamma:(\epsilon,\epsilon) \to \Sigma$ such that $\gamma(0)=q$ and $\dot{\gamma}(0)=e_i$ for $i=1,\ldots,d-1$. Then we have 
\begin{align*}
    D_{e_i} \pi(N_q)&=\dfrac{D}{ds}\Big|_{s=0} \pi(N_{\gamma(s)})=\sum_{j=1}^{d} \dfrac{d}{ds}\Big|_{s=0} \pi_j(N_{\gamma(s)}) \frac{\partial }{\partial x_j}\\
    &= \sum_{j=1}^{d} \nabla \pi_j(N_{q}) \dfrac{D}{ds}\Big|_{s=0} N_{\gamma(s)}  \frac{\partial }{\partial x_j}= (d \pi)_{N_q} D_{e_i} N_q.
\end{align*}
Moreover, since $d\pi$ is a symmetric matrix we gain 
\begin{equation}
\label{eq:FME1step}
H_{K,\Sigma}(q)=-\sum_{i=1}^{d-1} \escpr{(d \pi)_{N_q} D_{e_i} N_q ,e_i}=-\sum_{i=1}^{d-1} \escpr{ D_{e_i} N_q , (d \pi)_{N_q} e_i}.
\end{equation}
Since $\pi=\mathcal{N}_{K}^{-1}$ we obtain $d \pi=(d\mathcal{N}_{K})^{-1}$ and 
\[
e_i=d\mathcal{N}_{K}^{-1} d \mathcal{N}_K(e_i) = d\mathcal{N}_{K}^{-1}(k_i^{K}(\pi(N_q)) e_i)=k_i^{K}(\pi(N_q)) d \pi(e_i),
\]
by linearity. Therefore, we have $d \pi(e_i)=(k_i^K(\pi(N_q)))^{-1} e_i $. Hence, plugging this last equality in \eqref{eq:FME1step} we gain \eqref{eq:FMCEwrtE}.
\end{proof}
\end{proposition}

\subsubsection{The Ridge of the Finsler distance}\label{ridge}

In the previous section we obtained some regularity and geometric properties of $\df$ in a tubular neighborhood of $\partial\Om$. We shall see that some of these properties persist outside a tubular neighborhood.
We fix a convex body $K\in C^{\infty}_+$ and a bounded domain $\Om\subseteq\rr^d$ with $C^{2,1}$ boundary.
For any $p\in\Om$, we let $D(p):=\{q\in \partial\Om\,:\,\df(p)=\|p-q\|_{K}\}$.
Since $\df$ is continuous, then clearly $D(p)\neq\emptyset$ for any $p\in\Om$. Accordingly, we define the set 
\begin{equation}\label{om1}
    \Om_1:=\{p\in\Om\,:\,D(p)\text{ is a singleton}\},
\end{equation}
and we define the \emph{Ridge} of $\Om$ by
\begin{equation*}
    R:=\Om\setminus\intt{\Om_1}.\end{equation*}
We know, again thanks to \cite{MR2305073}, that, under our assumptions on $K$ and $\Om$,
\begin{equation}\label{diff}
    \df\in C^{2,1}(\intt{\Om_1}\cup\partial\Om).
\end{equation}
Moreover, in \cite[Corollary 1.6]{MR2094267} it is proven that the Hausdorff dimension of $R$ is at most $d-1$. This fact implies that $R$ has empty interior, so that 
\begin{equation}\label{bordo}
    \partial(\intt{\Om_1})=\partial\Om\cup R.
\end{equation}
The following result is inspired partially by \cite[Lemma 3.4]{MR872254}.
\begin{proposition}\label{utile}
Let $p\in\Om$, let $q\in D(p)$ 
and let 
\begin{equation*}
    (p,q):=\{tp+(1-t)q\,:\,t\in(0,1)\}.
\end{equation*}
Then $(p,q)\subseteq\intt{\Om_1}$ and
\begin{equation}\label{nx}
   D(\ga)=\{q\}
\end{equation}
for any $\ga\in(p,q)$. 
\end{proposition}
\begin{proof}
Let $p,q$ be as in the statement, and fix $\gamma\in (p,q)$. We already know that $D(\gamma)\neq\emptyset$. On the other hand, assume that there exists $q'\neq q$ such that $q'\in D(\gamma)$.
Let us notice that $p,q,q'$ cannot lie on the same line. Indeed, if by contradiction this was the case, then the only possibility is that $p$ is a convex combination of $\gamma$ and $q'$. But then the strict convexity of $K$ would imply that
\begin{equation*}
    \|\ga-q'\|_K\leq\|\ga-q\|_K<\|p-q\|_K\leq\|p-q'\|_K<\|\ga-q'\|_K,
\end{equation*}
which is absurd.
This in particular implies that $p,\gamma,q'$ do not lie on the same line. Therefore, thanks again to the strict convexity of $K$, we get that 
\begin{equation*}
    \|p-q'\|_{K}<\|p-\ga\|_{K}+\|\ga-q'\|_{K}\leq \|p-\ga\|_{K}+\|\ga-q\|_{K}=\|p-q\|_{K},
\end{equation*}
a contradiction to $q\in D(p)$. Hence \eqref{nx} is proved. Assume by contradiction that $\gamma\in R$. By Corollary 4.11 in \cite{MR2094267}, any point of the form $q+\lambda (\ga-q)$ with $\lambda>1$ has a point in $\partial\Om$ closer than $q$. On the other hand, taking $w$  the midpoint of $p$ and $\ga$, then by \eqref{nx} it holds that $D(w)=\{q\}$, which is impossible.
\end{proof}

Let us take a point $p\in\intt{\Om_1}$, and let $q\in D(p)$. Thanks to Proposition \ref{utile}, we know that 
\begin{equation*}
    \df(z)=\|z-q\|_K
\end{equation*}
for any $z$ in $(p,q)$. Recalling that $(p,q)\subseteq\intt{\Om_1}$, together with \eqref{diff}, and Proposition \ref{prop:findist} it is easy to see that
$\nabla\df(z)\neq 0$. Thus, at least locally, the level set $\Sigma_{d_K,\partial\Omega}(p)$ is a well-defined $C^2$ hypersurface. Reasoning as in Section \ref{Julian} we conclude that 
\begin{equation}\label{tange2}
-\divv(\pi_K\circ N_{d_K,\partial\Omega})(p)\geq H_{K_0,\partial\Om}(q)
\end{equation}
for any $p\in\intt{\Om_1}$, where $q\in D(p)$.

\subsection{The Heisenberg group}
\label{sc:heis}
Let $n\ge1$. We denote by $\mathbb{H}^n$ the Heisenberg group, defined as the $(2n+1)$-dimensional Euclidean space $\rr^{2n+1}$ endowed with the non-abelian group law $*$ given by 
\begin{equation}
\label{eq:Hnproduct}
 (x,y,t)*(\bar{x},\bar{y} ,\bar{t})=\left(x+\bar{x}, y+\bar{y},t+\bar{t}+ \sum_{i=1}^n\left(\bar{x}_iy_i-x_i\bar{y}_i\right) \right),   
\end{equation}
where $x=(x_1,\ldots,x_n)$, $\bar{x}=(\bar{x}_1,\ldots,\bar{x}_n)$, $y=(y_1,\ldots,y_n)$ and $\bar{y}=(\bar{y}_1,\ldots,\bar{y}_n)$.
A basis of left-invariant vector fields is given by 
\[
X_i=\dfrac{\partial}{\partial x_i} + y_i \dfrac{\partial}{\partial t}, \qquad Y_i=\dfrac{\partial}{\partial y_i} - x_i \dfrac{\partial}{\partial t}, \qquad T=\dfrac{\partial}{\partial t},
\]
for $i=1,\ldots,n$. 
For $p=(x,y,t) \in \hh^n$, the left translation by $p$ is the diffeomorphism $L_p(q) =p*q$.
We denote by $\mathcal{H}$ the horizontal distribution whose fiber at $p\in \mathbb{H}^n$ is the $2n$-dimensional space
\[
\mathcal{H}_p=\mathrm{span}\{X_i(p),\ Y_i(p)\ |\ i=1,\ldots, n\}.
\]
From now on we will always identify $\mathcal{H}_0$ with $\rr^{2n}$.
We shall consider on $\mathbb{H}^n$ the left-invariant Riemannian metric $g= \escpr{\cdot,\cdot}$, so that the vector fields $\{X_1,\ldots,X_n, Y_1,\ldots,Y_n, T\}$ form an orthonormal basis at every point, and we let $D$ be the Levi-Civita connection associated to the Riemannian metric $g$. 
The Riemannian volume of a set $E$ is, up to a constant, the Haar measure of the group and can be identified with the $(2n+1)-$dimensional Lebesgue measure. We denote it by $|E|$. The integral of a function $f$ with respect to the Riemannian measure is denoted by $\int f\, d\mathbb{H}^n$.
 
\subsection{Sub-Finsler norms and perimeter}
\label{sc:subfinsler}
Let $K_0\subseteq \mathcal{H}_0\equiv \rr^{2n}$ be a convex body in $C^2_+$ (cf. Subsection \ref{notsec}), $0\in \intt K_0$ and let $\|\cdot\|_{K_0}$ be the  associated norm in $\rr^{2n}$. In the following we shall write $\|\cdot\|,\|\cdot\|_{*}$ and $\pi$ instead of $\|\cdot\|_{K_0},\|\cdot\|_{K_0,*}$ and $\pi_{K_0}$ respectively. For any $p\in \mathbb{H}^n$, we define a left-invariant norm $\norm{\cdot}_p$ on $\mathcal{H}_p$ by means of the equality
\[
\norm{v}_{p}=\norm{dL_{p}^{-1}(v)}\quad v\in \mathcal{H}_p
\]
where $dL_p$ denotes the differential of $L_p$.
In particular, for a horizontal vector field $\sum_{i=1}^n f_iX_i+g_iY_i$ its norm at a point $p\in\mathbb{H}^n$ is given by 
\[
\Big\|\sum_{i=1}^n f_i(p)X_i (0)+g_i(p)Y_i(0)\Big\|=\|(f(p),g(p))\|,
\] 
where $f=(f_1,\ldots,f_n)$ and $g=(g_1,\ldots,g_n)$. Similarly, we extend the dual norm $\|\cdot\|_{*}$ and the projection $\pi$ to each fiber of the horizontal bundle.
When $\|\cdot\|$ is $C^l$ with $l\geq 2$, all norms $\|\cdot\|_p$ are $C^l$. Given a horizontal vector field $U$ of class $C^1$, we define $\pi(U)$ as the $C^1$ horizontal vector field satisfying 
\[
\|U\|_{*}=\langle U,\pi(U)\rangle.
\]
Proceeding as in \S~2.3 of \cite{PozueloRitore2021}, it is easy to see that the projection satisfies
\[
\pi\left(\sum_{i}^nf_iX_i+g_iY_i\right)=\mathcal{N}_{K_0}^{-1}\bigg(\frac{(f,g)}{\sqrt{|f|^2+|g|^2}}\bigg),
\]
where $|f|^2=\langle f,f\rangle$. 

\begin{definition}
Given a measurable set $E\subseteq \hh^n$ we say that $E$ has finite horizontal $K_0$-perimeter if
\[
P_{K_0,\mh}(E)= \sup \left\{ \int_E \divv(U) \ d\hh^n, U \in \mathcal{H}^1_0(\hh^n), \norm{U}_{K_0, \infty} \le 1 \right\} < + \infty,
\]
where $\mathcal{H}_0^1(\hh^n)$ is the space of $C^1$ horizontal compactly supported vector fields in $\hh^n$ and $\norm{U}_{K_0, \infty}=\sup_{p \in \hh^n} \norm{U_p}_p$. 
\end{definition}
\begin{remark}
The perimeter associated to the Euclidean norm $|\cdot|$ is the sub-Riemannian perimeter as it is defined in \cite{MR1871966, MR1404326}. A set has finite perimeter for a given norm if and only if it has finite perimeter for the standard sub-Riemannian perimeter. Hence all known results in the standard case apply to the sub-Finsler perimeter.
Moreover, if $E$ has $C^1$ boundary $\partial E$, then 
\[
P_{K_0, \mh}(E)=\int_{\partial E} \norm{N_h}_{*} d\sigma=: A_{K_0,\mh}(\partial E),
\]
where $N_h$ is the projection on the horizontal distribution $\mh$ of the Riemannian normal $N$ with respect to the metric $g$ and $d\sigma$ is the Riemannian measure of $\partial E$. For more details see \S~2.4 in \cite{PozueloRitore2021} when $n=1$.
\end{remark}
As a significant example, we consider a bounded open set $\Om\subseteq \rr^{2n}$ and a $C^1$ function $u:\Omega \to \rr$. Let $\text{Gr}(u)=\{(x,y,t) \in \hh^n \, : \, u(x,y)-t=0\}$ be the graph of $u$. Then we have 
\[
N_h=\dfrac{\sum_{i=1}^n (u_{x_i}-y)X_i +(u_{y_i}+x)Y_i}{\sqrt{1+ |\nabla u +F |^2}} \quad \text{and} \quad d \sigma= \sqrt{1+ |\nabla u +F |^2} \, dx dy,
\]
where $\nabla u(x,y)$ is the Euclidean gradient of $u(x,y)$ and $F(x,y)=(-y,x)$.
Therefore we get 
\[
A_{K_0, \mh}(\text{Gr}(u))= \int_{\Omega}  \norm{ \nabla u +F}_{*} \, dx dy.
\]

\subsection{The Sub-Finsler prescribed mean curvature equation}
Inspired by the previous computation and the sub-Riemannian problem studied by \cite{MR2262784} we consider the following problem. Let $\Omega \subseteq \rr^{2n}$ be a bounded open set and let $F \in L^1(\Omega,\rr^{2n})$, $\varphi\in W^{1,1}(\Om)$ and $H \in L^{\infty}(\Om)$. Then we set 
\begin{equation}
\label{eq:Ifun}
    \mathcal{I}(u)=\int_{\Omega} \|\nabla u +F\|_{*} \, dx dy + \int_{\Omega} Hu \, dx dy
\end{equation}
 for each $u\in W^{1,1}(\Om)$ such that  $u-\varphi \in W_{0}^{1,1}(\Omega)$.
We say that $u\in W^{1,1}(\Om)$ is a \textit{minimizer} for $\mathcal{I}$ if
\[
\mathcal{I}(u)\le \mathcal{I}(v)
\]
for all $v\in W^{1,1}(\Om)$ such that  $v-\varphi \in W_{0}^{1,1}(\Omega)$.
 In  \cite[Section 3]{MR2262784} the authors  investigate the first variation of the  functional $\mathcal{I}$ when $\| \cdot \|_{K_0,*}$ is the Euclidean norm $|\cdot|$, taking into account the bad beaviour of the singular set 
 \begin{equation}
 \label{eq:singularset}
     \Omega_0=\{(x,y) \in \Omega \ : \ (\nabla u +F)(x,y)=0 \}.
 \end{equation}
 In the next result we derive the Euler-Lagrange equation associated to $\mathcal{I}$ for $C^2$ minimizers.
 \begin{proposition}
Let $K_0$ be a $C_+^2$ convex body such that $0  \in \intt(K_0)$. Let $u \in C^2(\Om)$ be a minimizer for $\mathcal{I}$ defined in \eqref{eq:Ifun}. Assume that $F \in C^1(\Om,\rr^{2n})$. Let $\Om_0$ be the singular set defined in \eqref{eq:singularset}. Then $u$ satisfies 
\begin{equation}
\label{eq:HK0}
\begin{aligned}
\divv (\pi(\nabla u +F))=H\qquad\text{ in $\Omega\setminus\Omega_0$.}\\
\end{aligned}
\end{equation}
 \end{proposition}
 \begin{proof}
Given $v \in C^{\infty}_c(\Omega\setminus\Omega_0)$, by \cite[Lemma 3.2]{PozueloRitore2021}  the first variation is given by 
\begin{align*}
\dfrac{d}{ds}\Big|_{s=0} \mathcal{I}(u+sv)&=\int_{\Omega\setminus\Omega_0} \dfrac{d}{ds}\Big|_{s=0} \norm{\nabla (u+sv) +F}_{*}  \ dx dy + \int_{\Omega\setminus\Omega_0}H v \ dx dy\\
&=\int_{\Omega\setminus\Omega_0} \dfrac{d}{ds}\Big|_{s=0} \norm{\nabla u +F+ s\nabla v}_{*}  \ dx dy+ \int_{\Omega\setminus\Omega_0}H v \ dx dy\\
&= \int_{\Omega\setminus\Omega_0}  \escpr{ \nabla v, \pi(\nabla u +F))}  \ dx dy+ \int_{\Omega\setminus\Omega_0}H v \ dx dy\\
&=  \int_{\Omega\setminus\Omega_0}   v \, \left( H- \divv (\pi(\nabla u +F)) \right) \ dx dy.\qedhere
\end{align*}
 \end{proof}

\begin{remark}\label{reminfondo}
When $K_0$ is the unit disk $D_0 \subseteq \rr^{2n}$ centered at $0$ of radius $1$  we have 
\[
\pi_{D_0}(\nabla u +F)= \dfrac{\nabla u +F}{|\nabla u +F|}
\]
and \eqref{eq:HK0} is equivalent to 
\begin{equation*}
 \divv \left(\frac{\nabla u +F}{|\nabla u +F|}\right)=H.
\end{equation*}
\end{remark}

\section{The Finsler approximation problem}
\label{sc:Fap}
In this section we develop the Finsler approximation scheme in order to get rid of the singular nature of equation \eqref{eq:HK0}. To this aim, given $K_0$ a convex body in $C^{2}_+$ such that $0\in\intt{K_0}$ and $\varepsilon\in (0,1)$, we
 denote by $K_\eps$ the set
\begin{equation}
	\label{eq:Keps}
	K_\varepsilon:=\left\{(x,y,t)\in\rr^{2n+1} : \left(\frac{|t|}{\varepsilon}\right)^\frac{3}{2}+\|(x,y)\|^\frac{3}{2}\leq 1\right \}.
\end{equation}

Notice that $K_\eps \subseteq \rr^{2n+1} \equiv T_0 \hh^n$ (here $T_0 \hh^n$ denotes the tangent space of $\hh^n$ at $p=0$) is a strictly convex body with $0\in\intt(K_\eps)$. Moreover $\partial K_\eps$ is of class $C^1$. Indeed it is a level set of the $C^1$ function
$g_\eps(x,y,t):=\left(\frac{|t|}{\varepsilon}\right)^\frac{3}{2}+\|(x,y)\|^\frac{3}{2},$
whose gradient never vanishes on $\partial K_\eps$. Hence, the projection $\pi_{K_\eps}$ is well-defined and continuous. We shall write $\|\cdot\|_\eps$, $\|\cdot\|_{\eps,*}$ and $\pi_\eps$ instead of $\|\cdot\|_{K_\eps}$, $\|\cdot\|_{K_\eps,*}$ and $\pi_{K_\eps}$ respectively. The map $\pi_\eps^h$ is defined as the first $2n$ components of $\pi_\eps$. By abuse of notation, we write $\pi_\eps^h(x,y)=\pi_\eps^h(x,y,-1)$ when there is no confusion.

\begin{proposition}
	\label{prop:pieps2}
	Let $K_0$ be a convex body in $C^2_+$ such that $0\in\intt{K_0}$, and
	 let $K_\eps\subseteq \rr^{2 n+1}$ be the set defined in \eqref{eq:Keps}.	Then the following assertions hold:
	\begin{enumerate}

		\item[(i)] The map $\pi_\eps^h:\rr^{2n}\smallsetminus\{0\}\to\rr^{2n}$ satisfies
		$$\pi_\varepsilon^h(x,y)=\pi(x,y)\frac{\|(x,y)\|^2_{\ast}}{\left(\varepsilon^3+\|(x,y)\|_{\ast}^3\right)^\frac{2}{3}}.$$
		\item[(ii)] The map $\pi_ \eps^h$ can be extended to a $C^1$ map in $\rr^{2n}$ by setting $\pi^h_\eps(0,0)=(0,0)$.
		\item[(iii)] $\norm{(x,y,-1)}_{K_{\eps},*}=\left(\varepsilon^3+\|(x,y)\|_{\ast}^3\right)^\frac{1}{3}$.		
		
	\end{enumerate}
\end{proposition}

\begin{proof}
Let us prove that
\begin{equation}\label{eq:prepro}
\pi_\varepsilon(x,y,-1)=(\alpha\pi(x,y),-\varepsilon (1-\alpha^{3/2}))^{2/3}
\end{equation}
for some $0<\alpha(x,y)<1$. Given $(x,y)$ in $\rr^{2n}\setminus\{0\}$, we denote by $t_0$ the $(2n+1)$-th coordinate of $\pi_\varepsilon(x,y,-1)$ and we let $K_{t_0}\subseteq \mathbb{R}^{2n}$ be the convex set defined by $$K_{t_0}:=\{(x',y')\ : \ (x',y',t_0)\in K_\eps\}.$$ Then we have
\begin{equation*}
\begin{split}
    K_{t_0}\times\{t_0\}=\Bigl\{\Big(\frac{|t_0|}{\varepsilon}\Big)^\frac{3}{2}+\|(x',y')\|^{\frac{3}{2}}\leq 1  \Bigr\}
    =\left\{\|(x',y')\|\leq \left( 1-\left(\frac{|t_0|}{\varepsilon}\right)^\frac{3}{2}\right)^{\frac{2}{3}}  \right\}. 
    \end{split}
\end{equation*}
Hence it follows that $\pi_{t_0}=(1-(\frac{|t_0|}{\eps})^{\frac{3}{2}})^{\frac{2}{3}}\pi$. On the other hand, since $\pi_\eps$ is the inverse of the Gauss map, we can see that $(x,y,-1)$ is normal to $\partial K_\eps$ at $\pi_\eps(x,y,-1)$ and so $(x,y)$ is normal to $\partial K_{t_0}$ at $\pi^h_\eps(x,y)$, where $0<t_0<1$ satisfies $\|\pi_\eps^h(x,y)\|^{\frac{3}{2}}+(\frac{|t_0|}{\varepsilon})^\frac{3}{2}=1$. Since $K_{t_0}$ is strictly convex, the projection is unique and $\pi_\eps^h(x,y)=\pi_{t_0}(x,y)$. 
 Hence \eqref{eq:prepro} follows.
Taking the scalar product of $(x,y,-1)$  with the curve  $\beta(s)=(s\pi(x,y),-\varepsilon(1-s^{3/2})^{2/3})$, we get
\begin{equation*}
\langle(x,y,-1),\beta(s) \rangle=s\|(x,y)\|_{\ast}+\varepsilon (1-s^{3/2})^{2/3}.
\end{equation*}
Notice that $\beta$ is in $\partial K_\eps$ and $\beta(\alpha)$ is $\pi_\eps$. Hence in $s=\alpha$ the maximum of the scalar products of $(x,y,-1)$ with an element of $K_\varepsilon$ is attained. Thus we can take derivatives in $s=\alpha$, set them equal to $0$ and get
\[
0=\|(x,y)\|_{\ast}-\varepsilon\frac{\alpha^{\frac{1}{2}}}{(1-\alpha^{3/2})^{\frac{1}{3}}}.
\]
Then we obtain  \[
\alpha=\frac{\|(x,y)\|_{\ast}^2}{(\varepsilon^3+\|(x,y)\|_{\ast}^3)^{2/3}}\] and we get $(i)$. Since $\norm{(x,y,-1)}_{K_{\eps},*}=\langle(x,y,-1),\pi_\varepsilon(x,y,-1) \rangle$, a straightforward computation shows $(iii)$.
Finally, $(ii)$ follows from $(i)$ and the $2$-homogeneity of the map $\pi(\cdot)\|\cdot\|^2_{*}$.
\end{proof}

\begin{lemma} \label{lm:dernorm}
Let $u,v \in T_0 \hh^n$ and $s \in \rr$. Then we have 
\begin{equation}\label{eq:dernorm}
\dfrac{d}{ds}\Big|_{s=0} \norm{u+sv}_{\eps,*}= \escpr{v,\pi_{\eps}(u)}.
\end{equation}
\end{lemma}

\begin{proof}
Let $f(s)=\norm{u+sv}_{\eps,*}$ and $g(s)=\escpr{u+sv, \pi_{\eps}(u)}$. 
Notice that $f(s) \ge g(s)$  for each $s \in \rr$, since by definition  $\norm{u+sv}_{\eps,*}\ge \escpr{u+sv, \pi_\eps(u)}$ and  $f(0)=\norm{u}_{\eps,*}=\escpr{u,\pi_{\eps}(u)}=g(0)$. Therefore, by a standard argument $f'(0)=g'(0)$, and the thesis follows.
\end{proof}

Given a convex body $K_0\subseteq \mathbb{R}^{2n}$ in $C^2_+$ with $0\in\intt(K_0)$, and $K_\eps$ defined as in \eqref{eq:Keps}, we extend the reasoning of the previous section to define a left-invariant norm $\norm{\cdot}_{\eps}$ on $T\hh$ by means of the equality
\[
\Big\|\sum_{i=1}^n f_iX_i+g_i Y_i+hT \Big\|_{\eps, p}=\norm{(f(p),g(p),h(p))}_{\eps},
\] 
for any $p\in\hh^n$ with  $f=(f_1,\ldots,f_n)$ and $g=(g_1,\ldots,g_n)$. Again, $\|\cdot\|_{\eps,\ast}$ and $\pi_\eps$ can be extended to the tangent bundle in the usual way.
\begin{definition}
Given a measurable set $E\subseteq \hh^n$ we say that $E$ has finite $K_{\eps}$-perimeter if
\[
P_{K_{\eps}}(E)= \sup \left\{ \int_E \divv(U) \ d\hh^n, U \in \mathfrak{X}_0(\hh^n), \norm{U}_{K_{\eps}, \infty} \le 1 \right\} < + \infty,
\]
where $\norm{U}_{K_{\eps}, \infty}=\sup_{p \in \hh^n} \norm{U_p}_{\eps}$ and $\mathfrak{X}_0(\hh^n)$ is the space of $C^1$ compactly supported vector fields in $\hh^n$. 
\end{definition}

\begin{remark}
If $E$ has $C^1$ boundary $\partial E$, then 
\[
P_{K_{\eps}}(E)=\int_{\partial E} \norm{N}_{\eps,*} d\sigma= A_{\eps}(\partial E),
\]
where $N$ is the Riemannian normal with respect to the metric $g$ and $d\sigma$ is the Riemannian measure of $\partial E$.
Indeed by the divergence theorem we have 
\begin{align*}
P_{K_{\eps}}(E)&=\sup \left\{ \int_E \divv(U) \ d\hh^n, U \in \mathfrak{X}_0(\hh^n), \norm{U}_{K_{\eps}, \infty} \le 1 \right\}\\
&=\sup \left\{ \int_{\partial E} \escpr{U,N} \ d\hh^n, U \in \mathfrak{X}_0(\hh^n), \norm{U}_{K_{\eps}, \infty} \le 1 \right\}\\ &=\int_{\partial E} \norm{N}_{\eps,*} d\sigma,
\end{align*}
where the last equality can be proved proceeding exactly as in \cite{MR1871966, MR1404326}.
\end{remark}

\subsection{The Finsler prescribed mean curvature equation}
We are ready to derive the Finsler prescribed mean curvature equation, essentially in the same way as in the previous section. To this aim, let
$\Omega \subseteq \{t=0\}$ be a bounded open set and $u:\Omega \to \rr$ be a $C^2$ function. Then we have
\[
N=\dfrac{\sum_{i=1}^n (u_{x_i}-y)X_i +(u_{y_i}+x)Y_i -T}{\sqrt{1+ |\nabla u +F |^2}} \quad \text{and} \quad d \sigma= \sqrt{1+ |\nabla u +F |^2} \, dx dy,
\]
where $F(x,y)=(-y,x).$
Therefore we get 
\[
A_{K_{\eps}}(\text{Gr}(u))= \int_{\Omega}  \norm{(\nabla u +F,-1)}_{\eps,*} \, dx dy.
\]
Hence, inspired by this computation and thanks to Proposition \ref{prop:pieps2}, given $F \in L^1(\Omega,\rr^{2n})$, $\varphi\in W^{1,1}(\Om)$ and $H \in L^{\infty}(\Om)$, we define the approximating Finsler functional $\mathcal{I}_\eps$ by
\begin{equation}
\label{eq:Iepsfun}
    \mathcal{I}_\eps(u)=\int_{\Omega} \left(\varepsilon^3+\|(\nabla u+F)\|_{\ast}^3\right)^\frac{1}{3} \, dx dy + \int_{\Omega} Hu \, dx dy,
\end{equation}
for any $u\in W^{1,1}(\Om)$ such that $u-\varphi\in W^{1,1}_0(\Om)$. Arguing as in the previous section, and thanks to Lemma \ref{lm:dernorm}, we are able to deduce the Euler-Lagrange equation associated to \eqref{eq:Iepsfun}. Indeed, given $v \in C^{\infty}_c(\Omega)$, by Lemma \ref{lm:dernorm}, the first variation is given by:
\begin{align*}
\frac{d}{ds}\Big|_{s=0} \mathcal{I}_\eps(u+sv)&=\int_{\Omega} \dfrac{d}{ds}\Big|_{s=0} \norm{(\nabla (u+sv) +F,-1)}_{\eps,*}  \ dx dy+\int_\Om Hv\,dxdy\\
&=\int_{\Omega} \dfrac{d}{ds}\Big|_{s=0} \norm{(\nabla u +F,-1)+ s(\nabla v,0)}_{\eps,*}  \ dx dy+\int_\Om Hv\,dxdy\\
&= \int_{\Omega}  \escpr{ (\nabla v,0), \pi_{\eps}((\nabla u +F,-1))}  \ dx dy+\int_\Om Hv\,dxdy\\
&= \int_{\Omega}  \escpr{ \nabla v, \pi^h_{\eps}(\nabla u +F)}  \ dx dy+\int_\Om Hv\,dxdy\\
&=\int_{\Omega}   v (H-\divv (\pi^h_\eps(\nabla u +F)))\,dx dy.
\end{align*}
Then the Finsler prescribed mean curvature equation for the graph of $u$ is given by
\begin{equation}
\label{eq:HK}
\divv (\pi^h_{\eps}(\nabla u +F))=H\quad \mbox{in}\ \Omega.
\end{equation}
As already pointed out in the introduction, \eqref{eq:HK} is only degenerate elliptic in the singular set (cf. the computations of Section \ref{sc:aprioriestimate}). Therefore, in the next section, we will perturb \eqref{eq:HK} as in \eqref{findiv24intro} in order to apply the aforementioned classical Schauder fixed-point theory for elliptic equations.

\section{A priori estimates for the Finsler Prescribed Mean Curvature Equation}
\label{sc:aprioriestimate}
In this section we want to find classical solutions to the regularized Finsler approximating Dirichlet problem associated to \eqref{findiv24intro}, that is
\begin{equation}\label{riemreg}
 	\begin{cases}	\divv\left(\pi^h_{\eps}(\nabla u +F)\right)+\eta\divv\left(\frac{\nabla u+F}{\sqrt{1+|\nabla u+F|^2}}\right)=H&\text{ in  }\Om\\
		u=\varphi &\text{ in } \partial \Om,
	\end{cases}
 \end{equation}
where $\eps,\eta\in(0,1)$, $\Om\subseteq\rr^{2n}$ is a bounded domain with $C^{2,\alpha}$ boundary for $0<\alpha<1$, $K_0$ is a convex body in $C^{2,\alpha}_+$ with $0\in\intt{K_0}$, $H\in Lip(\overline\Om)$, $F =(F_1,\ldots,F_{2n})\in C^{1,\alpha}(\overline\Om,\rr^{2n})$ and $\varphi\in C^{2,\alpha}(\overline\Om)$. To this aim, let us fix some notation. It is easy to see that the map $G:\mathbb{R}^{2n}\setminus \{0\}\to \mathbb{R}^{2n}$ defined by $G(p)=\pi(p)\|p\|_{*}^2$ can be extended to a $2$-homogeneous and $C^1$ map setting $G(0)=0$. Moreover, for any $i=1,\ldots, 2n$
\begin{equation*}
	D_i(\|\cdot\|_{*}^3)=3G_i(\cdot),
\end{equation*}
where $G=(G_1,\ldots, G_{2n})$.
Thanks to Proposition \ref{prop:pieps2}, we can write the first equation of \eqref{riemreg} in the form
\begin{equation}\label{findiv24}
	\divv\left(\pi(\nabla u+F)\frac{\|\nabla u+F\|^2_{\ast}}{\left(\varepsilon^3+\|\nabla u+F\|_{\ast}^3\right)^\frac{2}{3}}\right)+\eta\divv\left(\frac{\nabla u+F}{\sqrt{1+|\nabla u+F|^2}}\right)=H.
\end{equation}
An easy computation yields 
\begin{equation*}
\label{eq:ddtG}
	\begin{split}
		\frac{1}{(\eps^3+\|\nabla u+F\|_{*}^3)^\frac{5}{3}}\big((\eps^3&+\|\nabla u+F\|_{*}^3)\divv(G(\nabla u+F))\\
		&-2G(\nabla u+F)(D^2u+DF)G(\nabla u+F)^T\big)\\
  &+\frac{\eta}{(1+|\nabla u+F|^2)^\frac{3}{2}}\left((1+|\nabla u+F|^2)\divv\big(\nabla u+F\right)\\
  &-(\nabla u+F)(D^2u+DF)(\nabla u+F)^T\big)=H.
	\end{split}
\end{equation*}
Therefore, we can write \eqref{findiv24} in the familiar form
\begin{equation*}
    \sum_{i,j=1}^{2n} A_{i,j}^{\eps,\eta}(z,\nabla u; F) D_{i,j} u+ B^{\eps,\eta}(z,\nabla u; F)=H,
\end{equation*}
where the coefficients $A_{i,j}^{\eps,\eta}$ and $B^{\eps,\eta}$ are defined by
\begin{equation}\label{coefficientsaaaaaaaa}
\begin{split}
    A_{i,j}^{\eps,\eta}(z,p; F)&:=\frac{1}{(\eps^3+\|p+F \|_*^3)^{\frac{2}{3}}} D_j G_i(p+F)-\frac{2}{(\eps^3+\|p+F \|_{*}^3)^{\frac{5}{3}}} G_i(p+F) G_j(p+F)\\
    &\quad+\frac{\eta}{\sqrt{1+|p+F|^2}}\delta_{ij}-\frac{\eta}{(1+|p+F|^2)^\frac{3}{2}}(p_i+F_i)(p_j+F_j)
\end{split}
\end{equation}
and
\begin{equation*}
    \begin{split}
         B^{\eps,\eta}(z,p; F) &:=\frac{1}{(\eps^3+\|p+F \|_*^3)^{\frac{2}{3}}} \sum_{i,j=1}^{2n} D_j G_i(p+F) D_i F_j\\
    &\quad-\frac{2}{(\eps^3+\|p+F \|_{*}^3)^{\frac{5}{3}}} G(p+F)\, DF \, G(p+F)^T\\
    &\quad+\frac{\eta}{\sqrt{1+|p+F|^2}}\divv F-\frac{\eta}{(1+|p+F|^2)^\frac{3}{2}}(p+F)DF(p+F)^T
    \end{split}
\end{equation*}
for any $z\in\Om$ and $p=(p_1,\ldots,p_{2n})\in\rr^{2n}$.
Therefore \eqref{findiv24} is a second-order quasi-linear equation. Moreover, thanks to the computations of the previous section and $(iii)$ in Proposition \ref{prop:pieps2}, we know that \eqref{findiv24} is the Euler-Lagrange equation associated to the functional
\[
u\mapsto \int_{\Omega}  \left(\eps^3+\norm{ \nabla u +F}_{*}^3\right)^{\frac{1}{3}}+\eta\sqrt{1+|\nabla u+F|^2} +uH \, dz.
\]
Notice that the matrix $A^{\varepsilon,\eta}$ is symmetric. Moreover, observing that
\begin{equation}\label{DG}
D_j (G_i(p))=\begin{cases}
2 \|p\|_{*} \pi_{j}(p)\pi_{i}(p)+ \|p\|_{*}^2 D_i \pi_{j}(p) & \text{if}\, p \ne0\\
0& \text{if}\, p =0,
\end{cases}
\end{equation}
we infer that \eqref{findiv24} is an elliptic equation. Indeed, assume first that $p+F=0$. Then, by \eqref{coefficientsaaaaaaaa} and \eqref{DG}
\begin{equation*}
    \sum_{i,j=1}^{2n}A_{i,j}^{\eps,\eta}(z,p; F)\xi_i\xi_j=\eta|\xi|^2
\end{equation*}
for any $\xi\in\rr^{2n}$. From the other hand, when $p+F\neq 0,$ \eqref{positivedefinite}, \eqref{DG} and the Cauchy-Schwarz inequality imply that
\begin{equation}\label{ellipticityconstant}
    \begin{split}
\sum_{i,j=1}^{2n}A_{i,j}^{\eps,\eta}&(z,p; F)\xi_i\xi_j=\sum_{i,j=1}^{2n}\frac{2\|p+F\|_*\pi_i(p+F)\pi_j(p+F)\xi_i\xi_j+\|p+F\|_*^2D_i\pi_j(p+F)\xi_i\xi_j}{(\eps^3+\|p+F \|_*^3)^{\frac{2}{3}}}\\
&\quad-\sum_{i,j=1}^{2n}\frac{2\|p+F\|_*^4\pi_i(p+F)\pi_j(p+F)\xi_i\xi_j}{(\eps^3+\|p+F \|_{*}^3)^{\frac{5}{3}}}+\eta\frac{(1+|p+F|^2)|\xi|^2-\langle p+F,\xi\rangle^2}{(1+|p+F|^2)^\frac{3}{2}}\\
&\geq \frac{\|p+F\|_*^2}{(\eps^3+\|p+F \|_{*}^3)^{\frac{2}{3}}}\left(\xi\,  D\pi(p+F)\,\xi^T\right)+\eta\frac{|\xi|^2}{(1+|p+F|^2)^\frac{3}{2}}\\
&>\eta\frac{|\xi|^2}{(1+|p+F|^2)^\frac{3}{2}}
    \end{split}
\end{equation}
for any $\xi\in\rr^{2n}$, so that we conclude that
\begin{equation}\label{actuallyelliptic}
\sum_{i,j=1}^{2n}A_{i,j}^{\eps,\eta}(z,p; F)\xi_i\xi_j\geq \frac{\eta}{(1+|p+F|^2)^\frac{3}{2}}|\xi|^2
\end{equation}
for any $z\in\overline\Om$ and any $p,\xi\in\rr^{2n}$. We remark that, by \eqref{ellipticityconstant}, equation \eqref{eq:HK} is elliptic outside the singular set.
In view of \eqref{actuallyelliptic}, we are in position to apply the classical theory for quasi-linear elliptic equations of \cite{GT}. In particular, we wish to rely on the following fundamental result, which is a direct consequence of \cite[Theorem 13.8]{GT} and subsequent remarks.

\begin{proposition}\label{gt}
\label{prop:Fex}
	Let $\Om\subseteq\rr^{2n}$ be a bounded domain with $C^{2,\alpha}$ boundary, for some $0<\alpha<1$, and let $\varphi\in C^{2,\alpha}(\overline\Om)$. Let us assume that $A_{i,j}^{\eps,\eta}(\cdot,\cdot;\sigma F),B^{\eps,\eta}(\cdot,\cdot;\sigma F)\in C^{\alpha}(\overline\Om\times\rr^{2n})$ for any $\sigma\in[0,1]$, and that the maps
	\begin{equation*}
	    \sigma\mapsto A_{i,j}^{\eps,\eta}(\cdot,\cdot;\sigma F),\quad\sigma\mapsto B^{\eps,\eta}(\cdot,\cdot;\sigma F)
	\end{equation*}
	are continuous as maps from $[0,1]$ to $C^{\alpha}(\overline\Om\times\rr^{2n})$.  If there exists a constant $M>0$ such that, for any $\sigma\in[0,1]$, any solution $u\in C^{2,\alpha}(\overline\Om)$ to the problem
	\begin{equation}\label{sigmadir}
		\begin{cases}\divv(\pi_\eps^h(\nabla u+\sigma F))+\eta\divv\left(\frac{\nabla u+\sigma F}{\sqrt{1+|\nabla u+\sigma F|^2}}\right)=\sigma H&\text{ in  }\Om\\
			u=\sigma\varphi &\text{ in } \partial\Om
		\end{cases}
	\end{equation}
	satisfies
	\begin{equation*}
		\|u\|_{C^{1}(\overline\Om)}\leq M,
	\end{equation*}
	then
	\begin{equation}\label{limiteq}
		\begin{cases}\divv(\pi_\eps^h(\nabla u+F))+\eta\divv\left(\frac{\nabla u+F}{\sqrt{1+|\nabla u+F|^2}}\right)= H&\text{ in  }\Om\\
			u=\varphi &\text{ in } \partial\Om
		\end{cases}
	\end{equation}
	admits a solution in $C^{2,\alpha}(\overline\Om)$.
\end{proposition}
\begin{remark}
Notice that the constant $M>0$ in Proposition \ref{prop:Fex} depends \emph{a priori}  on $\eps,\eta \in (0,1)$ and may blow up as $\eps,\eta \to 0$. However, in the sequel (cf. Propositions \ref{th:Cinfestimate}, \ref{prop:maxprinc} and  \ref{prop:boundgradest}) we will show that the estimates for the $C^1$ norm of solutions to
\eqref{sigmadir} can be made uniform in $\eps\in(0,1)$ and $\eta\in(0,\eta_0)$ for a sufficiently small constant $\eta_0\in (0,1)$. That would provide a constant $M>0$ 
 \emph{a posteriori} independent of $\eps$ and $\eta$, thus allowing to pass to the limit as $\eps,\eta \to 0$ (see Theorem \ref{th:main}).
\end{remark}
We shall need also the following weak maximum principle stated in \cite[Theorem 8.1]{GT}.
\begin{theorem}\label{8.1gt}
Let $\Om\subseteq\rr^d$ be a bounded domain. Let $L$ be the uniformly elliptic linear operator
\begin{equation*}
    \begin{split}
Lw&=\divv(a_{i,j}D_jw)+c_iD_i w
    \end{split}
\end{equation*}
where the coefficients $a_{i,j}$ and $c_i$
are bounded measurable functions on $\Om$. Let $w\in W^{1,2}(\Om)$ satisfy $Lw\geq 0$ in $\Om$ in distributional sense. Then
\[
\sup_{\Om} w\leq \sup_{\ptl \Om}w^+,
\]
where the value of $w^+=\max \{0,w\}$ in $\ptl\Om$ is understood in the sense of traces.
\end{theorem}
 First of all we need to guarantee the requested regularity for the coefficients of the equation. 
\begin{lemma}
\label{lm:holder}
 Let $K_0$ be a convex body in $C^{2,\alpha}_+$ with $0\in\intt{K_0}$.
 Let $F\in C^{1,\alpha}(\overline\Omega,\rr^{2n})$. Then there exists $0<\beta<1$ such that 
 $A_{i,j}^{\eps,\eta}(\cdot,\cdot;\sigma F),B^{\eps,\eta}(\cdot,\cdot;\sigma F)\in C^{\beta}(\overline\Om\times\rr^{2n})$ for any $\sigma\in[0,1]$. Moreover, the maps \begin{equation*}
	    \sigma\mapsto A_{i,j}^{\eps,\eta}(\cdot,\cdot;\sigma F),\quad\sigma\mapsto B^{\eps,\eta}(\cdot,\cdot;\sigma F)
	\end{equation*}
	are continuous as maps from $[0,1]$ to $C^{\beta}(\overline\Om\times\rr^{2n})$.
 \end{lemma}
\begin{proof}
The second statement easily follows from the definition of the coefficients. Let us prove the first statement.
It is clear, thanks to our assumptions on $K_0$ and $F$, that $A_{i,j}^{\eps,\eta}(\cdot,\cdot,\sigma F)$ and $B^{\eps,\eta}(\cdot,\cdot,\sigma F)$ belong to $C^{0}(\overline\Om\times\rr^{2n})$ for any $\sigma\in [0,1]$. Moreover, in view of \eqref{DG}, $D_jG_i$ is $C^{\alpha}(\rr^{2n} \setminus {0})$ for any $i,j=1,\ldots, 2n$, since $\partial K_0$ is $C^{2,\alpha}$. Finally, we get 
\[
\lim_{p \to 0} \dfrac{|D_j G_i|(p)}{|p|^{\alpha}}=0.
\]
Indeed, we have
\begin{align*}
\dfrac{|D_j G_i|(p)}{|p|^{\alpha}}&=2\dfrac{ \|p\|_{*}}{|p|^{\alpha}} |\pi_{j}(p)\pi_{i}(p)+ \|p\|_{*}^2 D_i \pi_{j}(p)|\\
&\le 2 \dfrac{\|p\|_{*}^{\alpha}}{|p|^{\alpha}} \|p\|_{*}^{1-\alpha} \left(\left|\pi_{j}(p)\pi_{i}(p)\right|+ \|p\|_{*} |D_i \pi_{j}(p)|\right) \\
&\le C\|p\|_{*}^{1-\alpha}  \to 0
\end{align*}
as $p \to 0$, since $\tfrac{\|p\|_{*}^{\alpha}}{|p|^{\alpha}}$ is bounded and the last factor in the previous inequality is $0$-homogeneous, thus in particular bounded. Then $D_j G_i$ belongs to $C^{\alpha}(\rr^{2n})$. Since $A^{\eps,\eta}_{i,j}$ and $B^{\eps,\eta}$ are obtained as composition, sum and product of H\"older functions, the conclusion follows.
\end{proof}
Therefore we are in position to apply Proposition \ref{gt}.
First of all we want to obtain estimates for the $C^0$ norm of solutions to \eqref{sigmadir}. In order to do this, inspired by \cite{MR336532}, we assume that there exists $\delta=\delta(K_0,\Om,H)\in (0,1]$ such that
\begin{equation}\label{in:hip}
\left|\int_\Om Hv dz\right|\leq (1-\delta)\int_{\Om}\|\nabla v\|_{*}dz 
\end{equation}
for any non-negative function $v\in C_c^\infty(\Omega)$. To justify this assumption, assume that we have a function $u\in C^2(\Omega)$ which solves \eqref{riemreg}.
Then, multiplying \eqref{riemreg} by a test function $v\in C^\infty_c(\Om)$, integrating over $\Om$ and letting $\eta\to 0$, by Proposition \ref{prop:pieps2} we get that
\begin{equation}\label{in:hip+}
\begin{split}
    \left|\int_\Om Hv\,dz\right|&\leq\left|\int_\Om v \divv(\pi^h_{\eps}(\nabla u +F))\,dz\right|+\eta\left|\int_\Om v\divv\Big(\frac{\nabla u +\sigma F}{\sqrt{1+|\nabla u + \sigma F|^2}}\Big)\,dz\right|\\
    &\leq\int_\Om  |\escpr{\pi^h_{\eps}(\nabla u +F),\nabla v}\mathcal|\, dz+\eta\int_\Om\left|\Big\langle\nabla v,\frac{\nabla u+\sigma F}{\sqrt{1+|\nabla u+\sigma F|^2}}\Big\rangle\right|\,dz\\
    &\leq \int_{\Om}  \|\nabla v\|_{*}dz+\eta\int_\Om|\nabla v|\,dz\\
    &\to \int_{\Om}  \|\nabla v\|_{*}dz.
\end{split}
\end{equation}
Notice that, as already pointed out in the introduction, \eqref{in:hip} is slightly stronger than \eqref{in:hip+}. 
 We begin by proving a technical lemma.
\begin{lemma}\label{tl}
Let $\sigma\in [0,1]$ and $\eps\in (0,1)$. Then
\begin{equation}\label{tecnical}
    \langle p,\pi_\eps^h(p+\sigma F)\rangle\geq \|p\|_\ast-1-\| F\|_*-\|-F\|_*
\end{equation}
for any $p\in\rr^{2n}$ and $z\in\overline\Om$.
\end{lemma}
\begin{proof}
Let us fix $z\in\overline\Om$ and $p\in\rr^{2n}$. If $p=0$ or $p+\sigma F=0$, then the assertion is trivial. Therefore, assume $p, p+\sigma F\neq 0$.
It is clear, recalling Proposition \ref{prop:pieps2} and using the Cauchy -Schwarz formula \eqref{in:C-S},  that 
\begin{equation*}\label{in:0.2}
\begin{split}
  \escpr{p,\pi_\eps^h(p+\sigma F)}&=\escpr{p+\sigma F,\pi_\eps^h(p+\sigma F)}-\escpr{\sigma F,\pi_\eps^h(p+\sigma F)}\\
  &\geq \frac{\|p+\sigma F\|_*^3}{(\eps^3+\|p+\sigma F\|_*^3)^\frac{2}{3}}-\left(\frac{\|p+\sigma F\|_*^3}{\eps^3+\|p+\sigma F\|_*^3}\right)^\frac{2}{3}\|\sigma F\|_{*}\\
  &\geq \frac{\|p+\sigma F\|_*^3}{(\eps^3+\|p+\sigma F\|_*^3)^\frac{2}{3}}-\| F\|_{*}.
\end{split}
\end{equation*}
Hence, noticing that
$$\|p+\sigma F\|_*\geq\|p\|_*-\|-\sigma F\|_*\geq \|p\|_*-\|- F\|_*$$
by the triangle inequality, it suffices to prove that
 \begin{equation}
 \label{eq:nvfU}
 \frac{\|p+\sigma F\|_*^3}{(\eps^3+\|p+\sigma F\|_*^3)^\frac{2}{3}}\geq \|p+\sigma F\|_{*}-1.
 \end{equation}
 When $\|p+\sigma F\|_{*}\leq 1$ \eqref{eq:nvfU} is trivial. Therefore let us assume $\|p+\sigma F\|_{*}>1$. 
Notice that \eqref{eq:nvfU} is equivalent to
 \begin{equation*}
  \|p+\sigma F\|_{*}^{\frac{9}{2}}\geq (\|p+\sigma F\|_{*}-1)^{\frac{3}{2}} (\varepsilon^3+\|p+\sigma F\|_{*}^3).
 \end{equation*}
Since $a^p-b^p\geq (a-b)^p$ when $ 0<b<a$ and $p>1$, it is enough to check that
\begin{equation*}\label{in:alt1}
    \begin{split}
        \|p+\sigma F\|_{*}^{9/2}&\geq (\|p+\sigma F\|_{*}^{3/2}-1) (\varepsilon^3+\|p+\sigma F\|_{*}^3)\\
  &=\varepsilon^3\|p+\sigma F\|_{*}^{3/2}+\|p+\sigma F\|_{*}^{9/2}-\varepsilon^3-\|p+\sigma F\|_{*}^3,
    \end{split}
\end{equation*}
which is clearly true since $\|p+\sigma F\|_{*}>1$ and $\varepsilon<1$.
\end{proof}
\begin{proposition}
\label{th:Cinfestimate}
     Let $\alpha\in (0,1)$ and $K_0$ be a convex body in $C^{2,\alpha}_+$ with $0\in\intt{K_0}$. Let $\Om\subseteq \rr^{2n}$ be a bounded open set, $\varphi\in C^2(\overline\Om)$, $H\in L^\infty(\Om)$ and $F\in C^0(\overline{\Omega},\rr^{2n})$. If condition \eqref{in:hip} is satisfied then there exist a constant $\eta_0=\eta_0(K_0,\delta)\in(0,1)$ and a constant $C_1=C_1(n,K_0,\Om,\varphi,F,\delta)>0$, independent of $\sigma\in[0,1]$, $\eps\in (0,1)$ and  $\eta\in (0,\eta_0)$, such that, for any solution  $u\in C^2(\overline\Omega)$ to \eqref{sigmadir} with $\eta\in(0,\eta_0)$ it holds that
\[
\|u\|_{L^{\infty}(\Om)}\leq C_1.
\]
\end{proposition}

\begin{proof}
Let us notice that \eqref{tecnical}, the equivalence between $\|\cdot\|_*$ and the Euclidean norm and the boundedness of $F$ allow to find constants $a_0,a_2>0$, independent of $\sigma\in [0,1]$ and $\eps\in (0,1)$, such that
\begin{equation*}
    \langle p,\pi_\eps^h(p+\sigma F)\rangle\geq a_0|p|-a_2
\end{equation*}
for any $z\in\Om$ and $p\in\rr^{2n}$. This fact, together with the boundedness of $H$, suggests to rely on \cite[Lemma 10.8]{GT} to limit ourselves to estimate $\|u\|_{L^1(\Om)}$. Indeed, it is not difficult to show that \cite[Lemma 10.8]{GT} remains true when condition $(10.23)$ of \cite{GT} allows a positive coefficient multiplying $|p|$. Moreover, its proof can be easily adapted to achieve estimates from above of $\sup_\Om -u$ in terms of $\|u^-\|_{L^1(\Om)}$ for any solution of  $Qu=0$ where $Q$ is defined in  $(10.5)$ of  \cite{GT}. In the end it suffices to estimate $\|u^+\|_{{L^1(\Om})}$ and $\|u^-\|_{{L^1(\Om})}$.
We only estimate $\|u^+\|_{L^1(\Om)}$, being the other case analogous. Moreover, up to replacing $u$ by $u-\|\varphi\|_{\infty,\partial\Om}$, we can assume that $u\leq 0$ in  $\partial \Om$.
Let us set $v=u^+$. Then it is clear that $v\in W^{1,\infty}(\Om)\cap W^{1,1}_0(\Om)$, and moreover $\nabla v$ exists in the classical sense for almost every $z\in\Om$. Therefore, since $u$ is in particular a weak solution to
\begin{equation*}
\divv(\pi_\eps^h(\nabla u+\sigma F))+\eta\divv\left(\frac{\nabla u+\sigma F}{\sqrt{1+|\nabla u+\sigma F|^2}}\right)=\sigma H,
	\end{equation*}
it follows that
\begin{equation}\label{eq:alt}
\begin{split}
\int_\Om \escpr{\nabla v,\pi_\eps^h(\nabla u+\sigma F)}+\eta\left\langle\nabla v,\frac{\nabla u+\sigma F}{\sqrt{1+|\nabla u+\sigma F|^2}}\right\rangle\,dz= -\int_\Om  v \sigma Hdz.
\end{split}
\end{equation}
We claim that 
\begin{equation}\label{in:0}
\escpr{\nabla v,\pi_\eps^h(\nabla u+\sigma F)}\geq \|\nabla  v\|_{*}-1 -\| F\|_{*}-\|-F\|_*
\end{equation}
holds in any point where $\nabla v$ exists in the classical sense. Indeed, in such points $\nabla v$ is either $0$ or $\nabla u$. In the first case \eqref{in:0} is trivial, while in the second case it follows from Lemma \ref{tl}.
It is well known that, since $v\geq 0$ and $v\in W^{1,1}_0(\Om)$, there exists a sequence of non-negative functions $(v_k)_k\subseteq C^\infty_c(\Om)$ converging to $v$ strongly in $W^{1,1}_0(\Om)$.  Moreover, thanks to \eqref{in:hip} it holds that
\begin{equation*}
\left|\int_\Om Hv_k dz\right|\leq (1-\delta)\int_{\Om}\|\nabla v_k\|_{*}\, dz.
\end{equation*}
Hence, passing to the limit in the previous equation, and recalling that $\|\cdot\|_*$ is equivalent to the Euclidean norm, we conclude that \eqref{in:hip} holds for $v$.
Combining this information with \eqref{eq:alt} and \eqref{in:0} we get that
\begin{equation*}
\begin{split}
0&=\int_\Om -\escpr{\nabla v,\pi_\eps^h(\nabla u+\sigma F)}-\eta\left\langle\nabla v,\frac{\nabla u+\sigma F}{\sqrt{1+|\nabla u+\sigma F|^2}}\right\rangle\,dz- \int_\Om v \sigma H\, dz\\
&\leq\int_\Om -\|\nabla  v\|_*+ 1+\| F\|_{*}+\|-F\|_*+\eta|\nabla v|\,dz+\left|\int_\Om v H\, dz\right| \\
&\leq \int_{\Om} -\|\nabla  v\|_*+ 1+\| F\|_{*}+\|-F\|_* +C\eta\|\nabla v\|_*+(1-\delta)\|\nabla v\|_{*}\,  dz\\
&=\int_{\Om} 1+\|F\|_{*}+\|-F\|_{*}+(C\eta -\delta)\|\nabla v\|_{*}\, dz,
\end{split}
\end{equation*} where $C=C(K_0)$ is a positive constant as in \eqref{in:norms}. Hence, choosing $\eta_0\in(0,1)$ such that $\delta-C\eta_0>0$, we conclude that
\begin{equation*}\label{in:0.1}
(\delta-C\eta_0)\int_\Om \|\nabla v\|_{*}\, dz\leq (\delta-C\eta)\int_\Om \|\nabla v\|_{*}\, dz\leq \int_\Om 1+\|F\|_{*} +\|-F\|_{*}\, dz
\end{equation*}
for any $\eta\in (0,\eta_0)$.
Thanks to the Poincaré inequality and the equivalence between $\|\cdot\|_{*}$ and the Euclidean norm, we conclude that there exists a constant $c_1$, independent of $\sigma\in[0,1]$, $\varepsilon\in(0,1)$ and $\eta\in (0,\eta_0)$, such that
\begin{equation*}
    \int_\Om u^+\,dz\leq c_1.
\end{equation*}
Since in the same way we can achieve an estimate for $u^-$, the thesis follows.
\end{proof}
The next step is to achieve gradient estimates, again in the $C^0$ norm, for solutions to \eqref{sigmadir}. As customary in this framework, we want to reduce ourselves to boundary gradient estimates via a suitable maximum principle. To this aim, arguing as in \cite{MR2262784}, we need to assume the existence of scalar functions $f_1,\ldots,f_{2n}\in C^1(\overline\Om)$ such that
\begin{equation}\label{taiwan}
    D_k F_i=D_i f_k\quad \mbox{for any}\quad  i,k=1,\ldots,2n. 
\end{equation}
 We stress that interior gradient estimates usually depend on the bounds of the coefficients and the ellipticity nature of the equation (cf. e.g. \cite[Chapter 15]{GT}). Consequently, since by \eqref{actuallyelliptic} the ellipticity constant tends to vanish as $\eta\to 0$,
the right way to achieve estimates which are uniform in $\epsilon,\eta\in (0,1)$ is to rely on a suitable maximum principle argument. 
 Indeed, thanks to \eqref{taiwan}, the following maximum principle, which is the Finsler counterpart of \cite[Proposition 4.3]{MR2262784}, holds. 
\begin{proposition}
\label{prop:maxprinc}
 Let $K_0$ be a convex body in $C^{2,\alpha}_+$ for $0<\alpha<1$ with $0\in\intt{K_0}$. Let $\Om\subseteq\rr^{2n}$ be a bounded domain. Let $F \in C^1(\Omega,\rr^{2n})$ be such that \eqref{taiwan} holds. Let $H$ be a constant.
Let $u\in C^2(\overline\Om)$ be a solution to \eqref{sigmadir}. Then 
\begin{equation}
	\|\nabla u\|_{\infty,\Om}\leq\|\nabla u\|_{\infty,\partial\Om}+2\|f\|_{\infty,\Om},
\end{equation}
where $f=(f_1,\ldots,f_{2n})$ is as in \eqref{taiwan}.
\end{proposition}
\begin{proof}
Fix $\sigma\in[0,1]$, $\eps\in (0,1)$ and $\eta\in(0,1)$. Let $v\in C^2_c(\Om)$ and fix $k\in\{1,\ldots,2n\}$. Then, multiplying \eqref{sigmadir} by $D_kv$, using Proposition \ref{prop:pieps2}, integrating over $\Om$, integrating by parts and exploiting the properties of $F$, it holds that
\begin{equation*}
	\begin{split}	0&=\int_{\Om}\left(\divv\left(\pi(\nabla u+\sigma F)\frac{\|\nabla u+\sigma F\|^2_{\ast}}{\left(\varepsilon^3+\|\nabla u+\sigma F\|_{\ast}^3\right)^\frac{2}{3}}+\eta \frac{\nabla u+\sigma F}{\sqrt{1+|\nabla u+\sigma F|^2}}\right)-\sigma H\right)D_kv\,dz\\
		&=\int_{\Om}\divv\left(\pi(\nabla u+\sigma F)\frac{\|\nabla u+\sigma F\|^2_{\ast}}{\left(\varepsilon^3+\|\nabla u+\sigma F\|_{\ast}^3\right)^\frac{2}{3}}+\eta \frac{\nabla u+\sigma F}{\sqrt{1+|\nabla u+\sigma F|^2}}\right)D_kv\,dz\\
		&=-\sum_{i=1}^{2n}\int_{\Om}\left(\pi_{i}(\nabla u+\sigma F)\frac{\|\nabla u+\sigma F\|^2_{\ast}}{\left(\varepsilon^3+\|\nabla u+\sigma F\|_{\ast}^3\right)^\frac{2}{3}}+\eta \frac{D_iu+\sigma F_i}{\sqrt{1+|\nabla u+\sigma F|^2}}\right)D_iD_kv\,dz\\
		&=-\sum_{i=1}^{2n}\int_{\Om}\left(\pi_{i}(\nabla u+\sigma F)\frac{\|\nabla u+\sigma F\|^2_{\ast}}{\left(\varepsilon^3+\|\nabla u+\sigma F\|_{\ast}^3\right)^\frac{2}{3}}+\eta \frac{D_iu+\sigma F_i}{\sqrt{1+|\nabla u+\sigma F|^2}}\right)D_kD_iv\,dz\\
		&=\sum_{i=1}^{2n}\int_{\Om}D_k\left(\pi_{i}(\nabla u+\sigma F)\frac{\|\nabla u+\sigma F\|^2_{\ast}}{\left(\varepsilon^3+\|\nabla u+\sigma F\|_{\ast}^3\right)^\frac{2}{3}}+\eta \frac{D_iu+\sigma F_i}{\sqrt{1+|\nabla u+\sigma F|^2}}\right)D_iv\,dz\\
		&=\sum_{i,j=1}^{2n}\int_{\Om}A_{i,j}^{\varepsilon,\eta}(z,\nabla u;\sigma F)D_k(D_ju+\sigma F_j)D_iv\,dz\\
		&=\sum_{i,j=1}^{2n}\int_{\Om}A_{i,j}^{\varepsilon,\eta}(z,\nabla u;\sigma F)D_j(D_ku+\sigma f_k)D_ivdz,
	\end{split}
\end{equation*}

being $A_{i,j}^{\eps,\eta}$ as in \eqref{coefficientsaaaaaaaa}.
Therefore we proved that 
\begin{equation}\label{mp2}
	\sum_{i,j=1}^{2n}\int_{\Om}A_{i,j}^{\varepsilon,\eta}(z,\nabla u;\sigma F)D_j(D_ku+\sigma f_k)D_iv\,dz=0
\end{equation}
for any $v\in C^2_c(\Om)$. Arguing as in \cite[Proposition 4.3]{MR2262784} it is easy to show that \eqref{mp2} actually holds for any $v\in C_c^1(\Om)$. Therefore, recalling \eqref{actuallyelliptic}, we proved that $D_ku+\sigma f_k$ is a weak solution to the linear uniformly elliptic equation
$$\divv(a_{i,j}^{\varepsilon,\eta}D_jw)=0,$$
where
$$a_{i,j}^{\varepsilon,\eta}(z):=A_{i,j}^{\varepsilon,\eta}(z,\nabla u;\sigma F(z)).$$
Hence, being $a_{i,j}^{\varepsilon,\eta}(z)$ bounded in $\Om$, thanks to Theorem \ref{8.1gt} with $b_i,c_i,d=0$ we conclude that 
\begin{equation*}
	\|\nabla u+\sigma f\|_{\infty,\Om}\leq\|\nabla u+\sigma f\|_{\infty,\partial\Om},
\end{equation*}
which in particular implies that 
\begin{equation}
	\|\nabla u\|_{\infty,\Om}\leq\|\nabla u\|_{\infty,\partial\Om}+2\|f\|_{\infty,\Om}.\qedhere
\end{equation}
\end{proof}

Finally we are left to provide boundary gradient estimates for solutions to \eqref{sigmadir}.
Therefore, inspired by \cite{MR336532}, we have to impose some constraints on the values of $H$ depending on the Finsler mean curvature of $\partial\Om$. More precisely, we require that
\begin{equation}
\label{curvcond}
|H|(z_0)< H_{K_0,\partial\Om}(z_0)
\end{equation}
for any $z_0\in\partial\Om$, where $H_{K_0,\partial\Om}$ is the $K_0$-mean curvature as defined in Subsection \ref{subjulian}. Here and in the rest of this section we assume that $K_0$ is a convex body in $C^{\infty}_+$ such that $0 \in \intt{K_0}$, since we need to apply the results of Section \ref{Julian} and Section \ref{ridge}.

\begin{proposition}
\label{prop:boundgradest}
 Let $K_0$ be a convex body in $C^{\infty}_+$ with $0\in\intt{K_0}$. Let $\Om\subseteq\rr^{2n}$ be an open and bounded set with $C^{2,\alpha}$ boundary, for some $0<\alpha<1$. Let $\varphi\in C^2(\overline\Om)$, $F\in C^0(\overline\Om,\rr^{2n})$ and $H\in Lip(\Om)$ satisfying \eqref{curvcond}. Finally, assume that there exist a constant $\eta_0=\eta_0(n,K_0,\Om,\varphi,F,H)\in(0,1)$ and a constant $\tilde{C_1}=\tilde{C_1}(n,K_0,\Om,\varphi,F,H)>0$, independent of $\sigma\in[0,1]$, $\eps\in (0,1)$ and $\eta\in (0,\eta_0)$, such that, for any solution $u\in C^2(\overline\Om)$ to \eqref{sigmadir} it holds that
\begin{equation}\label{uniformbound}
	\|u\|_{\infty,\Om}\leq \tilde{C_1}.
\end{equation}
Then, up to choosing a smaller $\eta_0=\eta_0(n,K_0,\Om,\varphi,F,H)\in(0,1)$, there exist a constant $C_2=C_2(n,K_0,\Om,\varphi,F,\tilde{C_1},H)>0$, independent of $\sigma\in[0,1]$, $\eps\in (0,1)$ and $\eta\in (0,\eta_0)$, such that any solution $u\in C^2(\overline\Om)$ to \eqref{sigmadir} with $\eta\in (0,\eta_0)$ satisfies
\begin{equation}
	\|\nabla u\|_{\infty,\partial\Om}\leq C_2.
\end{equation}
\end{proposition}
\begin{proof}
First of all we notice that, being $\partial\Om$ compact and $H_{K_0,\partial\Om}$ continuous,
\eqref{curvcond} implies the existence of a positive constant $C_3=C_3(K_0,\Om,H)$ such that
\begin{equation}
\label{curvcond2}
    |H(z_0)|\leq H_{K_0,\partial\Om}(z_0)-3C_3
\end{equation}
for any $z_0\in\partial\Om$.
 In order to prove this result we use a barrier argument as in \cite[Chapter 14]{GT}. Therefore, for any $z_0\in\partial\Om$, we have to find a neighborhood $\mathcal{N}$ of $z_0$ in $\Om$ and two functions $w^+,w^-\in C^2(\mathcal{N})$, called \emph{upper barrier} and \emph{lower barrier} respectively, such that
	\begin{equation*}\label{zero}
		w^+(z_0)=w^-(z_0)=\sigma\varphi(z_0),
	\end{equation*}
	\begin{equation*}\label{diseqbordo}
		w^-(z)\leq u(z)\leq w^+(z)
	\end{equation*}
	for any $z\in\partial\mathcal{N}$, 
	\begin{equation*}\label{superbar}
		\divv(\pi_\eps^h(\nabla w^++\sigma F))+\eta\divv\left(\frac{\nabla w^++\sigma F}{\sqrt{1+|\nabla w^++\sigma F|^2}}\right)<\sigma H
	\end{equation*}
	for any $z\in\mathcal{N}$ and 
	\begin{equation*}\label{subbar}
		\divv(\pi_\eps^h(\nabla w^- + \sigma F))+\eta\divv\left(\frac{\nabla w^-+\sigma F}{\sqrt{1+|\nabla w^-+\sigma F|^2}}\right) >\sigma H
	\end{equation*}
	for any $z\in\mathcal{N}$.
In this proof we deal only with the upper barrier, being the other case analogous.
In order to find an upper barrier, we consider a tubular neighborhood $\mathcal{O}$ of $\partial\Om$ and we let $\Gamma_\mu:=\{x\in\overline\Om\,:\,\dfz(x)<\mu\}$, where $\dfz$ is the Finsler distance from the boundary, $\mu\in (0,\mu_0)$ and $\mu_0>0$ is small enough to ensure that $\Gamma_\mu\subseteq\Gamma_{\mu_0}\Subset\mathcal{O}$ for any $\mu\in (0,\mu_0)$. Let us denote by $H_{\Sigma_{d(z)}}(z)$ the Euclidean mean curvature of $\Sigma_{d(z)}$ at any $z\in\overline\Gamma_{\mu_0}$. Being $H_{\Sigma_{d(z)}}$ continuous on $\overline\Gamma_{\mu_0}$, there exists a constant $C_4=C_4(\Om,K_0)>0$ such that
\begin{equation}\label{boundeuclcurv}
    |H_{\Sigma_{d(z)}}(z)|\leq C_4
\end{equation}
for any $z\in\overline\Gamma_{\mu_0}$.
We fix $\mu\in(0,\mu_0)$ and we define $w^+:\Gamma_\mu\scu\rr$ by $w^+(z):=k\dfz(z)+\sigma\varphi(z)$, where $k>0$ has to be chosen.
First, thanks to \eqref{diff}, $w^+\in C^2(\overline{\Gamma_\mu})$, and for any $z\in\Gamma_\mu$ there exists a unique $z_0\in\partial\Om$ such that $\dfz(z)=\|z-z_0\|$. Moreover, it is clear that $w^+(z_0)=\sigma\varphi(z_0)$ for any $z_0\in\partial\Om$. Thanks to \eqref{uniformbound}, if we choose $$k\geq \frac{\tilde{C_1}+\|\varphi\|_{\infty,\Om}}{\mu},$$ it follows that $w^+(z)\geq u(z)$ for any $z\in\Om$ with $\dfz(z)=\mu$, and so we conclude that $u(z)\leq w^+(z)$ for any $z\in\partial\Gamma_\mu$.
We are left to show that $w^+$ is a subsolution to \eqref{sigmadir}. Therefore it suffices to show that
\begin{equation*}
    (\eps^3+\|\nabla w^++\sigma F\|_{*}^3)^\frac{5}{3}
		\left(\divv(\pi_\eps^h(\nabla w^++\sigma F))+\eta\divv\Big(\frac{\nabla w^++\sigma F}{\sqrt{1+|\nabla w^++\sigma F|^2}}\Big)-\sigma H\right) < 0
\end{equation*}
on $\Gamma_\mu$. Taking $k>\sup_\Om\|-F\|_{*},$ \eqref{eq:Ei} ensures that $k\nabla \dfz(z)+\sigma F(z)\neq 0$ for any $z\in\Gamma_\mu$ and $\sigma\in [0,1]$. Let us notice that Proposition \ref{prop:pieps2} and a simple computation imply that
\begin{equation*}
    \begin{split}
        (\eps^3+&\|\nabla w^++\sigma F\|_{*}^3)^\frac{5}{3}
		\divv(\pi_\eps^h(\nabla w^++\sigma F))\\
		=&(\eps^3+\|\nabla w^++\sigma F\|_{*}^3)^\frac{5}{3}
		\divv\left(\frac{\pi(\nabla w^++\sigma F)\|\nabla w^++\sigma F\|_*^2}{(\eps^3+\|\nabla w^++\sigma F\|_*^3)^\frac{2}{3}}\right)\\
		=&(\eps^3+\|\nabla w^++\sigma F\|_{*}^3)\underbrace{\divv(\pi(\nabla w^++\sigma F)\|\nabla w^++\sigma F\|_*^2)}_A\\
		&+(\eps^3+\|\nabla w^++\sigma F\|_{*}^3)^\frac{5}{3}\underbrace{\|\nabla w^++\sigma F\|_*^2\langle \pi(\nabla w^++\sigma F),\nabla\left( (\eps^3+\|\nabla w^++\sigma F\|_*^3)^{-\frac{2}{3}}\right)\rangle}_B.
    \end{split}
\end{equation*}
We estimate separately $A$ and $B$.  In the following computations we let $d:=\dfz$ and $R_\sigma :=\sigma\nabla\varphi+\sigma F$. We are going to exploit the fact that, thanks to the homogeneity properties of the equation, the contribution of $R_\sigma$ as $k\to\infty$ is negligible. Let us notice that by \eqref{eq:Ei} and \eqref{eq:normpi} we get 
\begin{equation}\label{finseikcons2}
	\pi(\nabla \dfz)\cdot D^2\dfz=0.
\end{equation}
Hence, thanks to \eqref{eq:Ei}, \eqref{finseikcons2}, the $1$-homogeneity of $\|\cdot\|_*$, the $0$-homogeneity of $\pi$, the $-1$-homogeneity of $D\pi$ and the properties of $\|\cdot\|_{*}$, it holds that
\begin{equation*}
	\begin{split}
		A=&\|k\nabla d+R_\sigma\|^2_{*}\sum_{i=1}^{2n}D_i\left(\pi_{i}(k\nabla d+R_\sigma)\right)+\sum_{i=1}^{2n}\pi_{i}(k\nabla d+R_\sigma)D_i\left(\|k\nabla d+R_\sigma\|^2_{*}\right)\\
		=&\|k\nabla d+R_\sigma\|^2_{*}\sum_{i,j=1}^{2n}D_i\pi_{j}(k\nabla d+R_\sigma)(kD_{ij}d+D_i R_{\sigma,j})\\
		&+2\|k\nabla d+R_\sigma\|_{*}\pi(k\nabla d+R_\sigma)\cdot(kD^2 d+DR_\sigma)\cdot\pi(k\nabla d+R_\sigma)^T\\
		=&k^2\left\|\nabla d+\frac{R_\sigma}{k}\right\|^2_{*}\sum_{i,j=1}^{2n}D_i\pi_{j}\left(\nabla d+\frac{R_\sigma}{k}\right)\left(D_{ij}d+\frac{D_i R_{\sigma,j}}{k}\right)\\
		&+2k^2\left\|\nabla d+\frac{R_\sigma}{k}\right\|_{*}\pi\left(\nabla d+\frac{R_\sigma}{k}\right)\cdot\left(D^2 d+\frac{DR_\sigma}{k}\right)\cdot\pi\left(\nabla d+\frac{R_\sigma}{k}\right)^T\\
		=&k^2(1+o(1))(\divv(\pi(\nabla d))+o(1))+2k^2(1+o(1))(\pi(\nabla d)\cdot D^2 d\cdot \pi(\nabla d)^T+o(1))\\
		=&k^2\divv(\pi(\nabla d))+o(k^2),
	\end{split}
\end{equation*}
which allows to infer that
\begin{equation*}
    \begin{split}
        (\eps^3+\|\nabla w^++\sigma F\|_{*}^3)A=k^5\divv(\pi(\nabla d))+o(k^5)
    \end{split}
\end{equation*}
as $k\to\infty$, where $o(k^2)$ is uniform with respect to $z\in\Gamma_\mu$, $\eps\in (0,1)$ and $\sigma\in[0,1]$. Now, exploiting the same properties as above, we estimate $B$:
\begin{equation*}
    \begin{split}
        (\eps^3+\|k\nabla d+R_\sigma\|_*^3)^{\frac{5}{3}}B&=-2\|k\nabla d+R_\sigma\|_*^4\langle \pi(k\nabla d+R_\sigma),\nabla(\|k\nabla d+R_\sigma\|_*)\rangle\\
        &=-2\|k\nabla d+R_\sigma\|_*^4\pi(k\nabla d+R_\sigma)\cdot(kD^2 d+DR_\sigma)\cdot\pi(k\nabla d+R_\sigma)^T\\
        &=-2k^5\left\|\nabla d+\frac{R_\sigma}{k}\right\|_*^4\pi\left(\nabla d+\frac{R_\sigma}{k}\right)\cdot\left(D^2 d+\frac{DR_\sigma}{k}\right)\cdot\pi\left(\nabla d+\frac{R_\sigma}{k}\right)^T\\
        &=-2k^5(1+o(1))(\pi(\nabla d)\cdot D^2 d\cdot \pi(\nabla d)^T+o(1))\\
        &=-2k^5(1+o(1))o(1)\\
        &=o(k^5).
    \end{split}
\end{equation*}
as $k\to\infty$ and uniformly with respect to $\eps\in(0,1)$, $\sigma\in[0,1]$ and $z\in\Gamma_\mu$.
From a similar computation, it follows that
\begin{equation*}
\begin{split}
\divv\left(\frac{\nabla w^++\sigma F}{\sqrt{1+|\nabla w^++\sigma F|^2}}\right)&=\frac{\divv(\nabla d)+\frac{\divv R_\sigma}{k}}{\sqrt{\frac{1}{k^2}+\left|\nabla d+\frac{R_\sigma}{k}\right|^2}}-\frac{\left(\nabla d+\frac{R_\sigma}{k}\right)\cdot\left( D^2d+\frac{DR_\sigma}{k}\right)\cdot \left(\nabla d+\frac{R_\sigma}{k}\right)^T}{\left(\frac{1}{k^2}+\left|\nabla d+\frac{R_\sigma}{k}\right|^2\right)^{3/2}}\\
&=\frac{\divv(\nabla d)}{|\nabla d|}-\frac{\nabla d\cdot D^2 d\cdot \nabla d^T}{|\nabla d|^3} +o(1)\\
&=\divv\left(\frac{\nabla d}{|\nabla d|}\right)+o(1)
\end{split}
\end{equation*}
as $k\to\infty$ and uniformly with respect to $\sigma\in[0,1]$ and $z\in\Gamma_\mu$.
Finally, it is easy to see that 
\begin{equation*}
    -(\eps^3+\|\nabla w^++\sigma F\|_{*}^3)^\frac{5}{3}\sigma H\leq(\eps^3+\|\nabla w^++\sigma F\|_{*}^3)^\frac{5}{3}|H|=k^5|H|+o(k^5)
\end{equation*}
as $k\to\infty$ and uniformly with respect to $\eps\in(0,1)$, $\sigma\in[0,1]$ and $z\in\Gamma_\mu$.
In the end we get that 
\begin{equation*}
\begin{split}
      (\eps^3+\|\nabla w^++\sigma F\|_{*}^3)^\frac{5}{3}
		&\left(\divv(\pi_\eps^h(\nabla w^++\sigma F))+\eta \divv\left(\frac{\nabla w^++\sigma F}{\sqrt{1+|\nabla w^++\sigma F|^2}}\right)-\sigma H\right)\\
  &\leq k^5\left(\divv(\pi(\nabla d))+\eta\divv\left(\frac{\nabla d}{|\nabla d|}\right)+|H|\right)+o(k^5)
\end{split}
\end{equation*}
as $k\to\infty$ and uniformly with respect to $\eps\in(0,1)$, $\sigma\in[0,1]$ and $z\in\Gamma_\mu$.

 Now, let $z\in\Gamma_\mu$ and let $z_0\in\partial\Om$ be such that $d(z)=\|z-z_0\|$. Thanks to the Lipschitz continuity of $H$ and the equivalence between $\|\cdot\|$ and the Euclidean norm, there exists a constant $C_5=C_5(K_0)$ such that 
\begin{equation}\label{Hnocost}
    \begin{split}
        |H|(z)=|H|(z_0)+|H|(z)-|H|(z_0)
        \leq |H|(z_0)+C_5 d(z)
        \leq |H|(z_0)+C_5\mu.
    \end{split}
\end{equation}
Hence, thanks to \eqref{eq:meanpar}, \eqref{tange}, \eqref{curvcond2} and \eqref{boundeuclcurv}, we conclude that 
\begin{equation}\label{perdopo}
    \begin{split}
        \divv(\pi(\nabla d))(z)+\eta\divv\left(\frac{\nabla d}{|\nabla d|}\right)+|H|(z_0)+C_5\mu
        &=-H_{K_0,\Sigma_{d(z)}}(z)-\eta H_{\Sigma_{d(z)}}(z)+|H|(z_0)+C_5\mu\\
        &\leq-H_{K_0,\partial\Om}(z_0)+\eta C_4+|H|(z_0)+C_5\mu\\
        &\leq-C_3<0,
    \end{split}
\end{equation}
provided that $\mu\leq \frac{C_3}{C_5}$ and $\eta\leq  \frac{C_3}{C_4}$.
Hence we found an upper barrier, from which the thesis follows.
\end{proof}

\begin{remark}
\label{rk:convex}
Assume that $n=1$, let $\Omega \subseteq \rr^2$ and $K_0 \in C^2_{+}$ be a convex body of $\rr^2$. 
If \eqref{curvcond} holds then $\Omega$ is strictly convex. 
Indeed, by Proposition \ref{prop:FMCEwrtE} we have
\[
0 \leq |H|< -\dfrac{\escpr{D_{e_1} N_{z_0}, e_1}}{k^{K_0}(\pi(N_{z_0}))}=\dfrac{k^{\partial \Om}(z_0)}{k^{K_0}(\pi(N_{z_0}))},
\]
where $k^{K_0}$ and $k^{\partial \Om}$ are the the Euclidean geodesic curvatures of $\partial K $ and $\partial \Om$. Since $k^{K_0}$ is strictly positive we obtain $k^{\partial \Om}(z_0)>0$, hence $\Om$ is strictly convex.

\end{remark}

To conclude this section, inspired by \cite{MR282058} we want to show that, in the particular case in which $H$ is constant and $n=1$, then we can exploit \eqref{curvcond} in order to obtain uniform estimates of the function, without requiring the validity of \eqref{in:hip}. Again, in order to apply the results of Section \ref{Julian} and Section \ref{ridge}, we assume that $K_0$ is a convex body in $C^{\infty}_+$ such that $0 \in \intt{K_0}$ and $\partial \Omega$ belongs to $C^{2,1}$.
\begin{proposition}
\label{th:Cinfestimateconst}
Assume that $n=1$. Let $K_0$ be a convex body in $C^{\infty}_+$ with $0\in\intt{K_0}$. Let $\Om\subseteq \rr^{2}$ be a bounded domain with $C^{2,1}$ boundary, let $\varphi\in C^2(\overline\Om)$ and let $H$ be a constant which satisfies \eqref{curvcond}. There exists a constant $C_1=C_1(K_0,\Om,\varphi,H,F)>0$, independent of $\sigma\in[0,1]$, $\eps\in (0,1)$ and $\eta\in (0,1)$, such that, for any solution  $u\in C^2(\overline\Omega)$ to \eqref{sigmadir},
it holds that
\[
\|u\|_{\infty,\Om}\leq C_1.
\]
\end{proposition}
\begin{proof} Let $k^{K_0}$ be the geodesic curvature of $K_0$. Since $K_0\in C_+^\infty$, then in particular
$k^{K_0}(p)>0$ for any $p\in \partial K_0$. Let $C_3=C_3(K_0,\Om,H)$ be as in \eqref{curvcond2}.
Let us define the function $v:\intt{\Om_1}\scu\rr$ by
\begin{equation}\label{barrierafinale}
    v(z):=\sup_{\partial\Om}|\varphi|+k\dfz(z)
\end{equation}
for any $z\in\Om_1$, where $k>0$ has to be chosen and $\Om_1$ is the set defined in \eqref{om1}. We already know (cf. \eqref{diff}) that $v\in C^2(\intt{\Om_1})$.  We repeat \emph{verbatim} the computations of the proof of Proposition \ref{prop:boundgradest} up to \eqref{perdopo}, with the difference that, being $H$ constant, we can choose $C_5=0$ in \eqref{Hnocost}. Since $n=1$, we exploit Proposition \ref{prop:FMCEwrtE} to infer that
\begin{equation}\label{perdopo2.0}
    \begin{split}
        \divv(\pi(\nabla d))(z)+\eta\divv\left(\frac{\nabla d}{|\nabla d|}\right)+|H|&
        =-H_{K_0,\Sigma_{d(z)}}(z)-\eta H_{\Sigma_{d(z)}}(z)+|H|\\
        &=-H_{K_0,\Sigma_{d(z)}}(z)-\eta k^{K_0}\left(\pi_K(N_z)\right) H_{K_0,\Sigma_{d(z)}}(z)+|H|\\
        &\leq-H_{K_0,\partial\Om}(z_0)+|H|\\
        &\leq-3C_3<0.
    \end{split}
\end{equation}
Hence there exists $k>0$, independent of $\eps\in(0,1),$ $\eta\in (0,1)$, $\sigma\in[0,1]$ and $z\in\Om_1$, such that $v$ is a subsolution to \eqref{sigmadir} on $\intt{\Om_1}$. Therefore, arguing as in the proof of \cite[Theorem 10.7]{GT}, it follows that $w:=u-v$ is a weak supersolution on $\intt{\Om_1}$ to a linear elliptic equation of the form
\begin{equation*}
    \sum_{i,j=1}^{2n}D_i(a_{i,j}(z)D_{j}w(z))+\sum_{i=1}^{2n}c_i(z)D_iw(z)=0.
\end{equation*}
Hence, thanks to Theorem \ref{8.1gt} and recalling \eqref{bordo}, it follows that 
\begin{equation*}
    \sup_{\Om_1}(u-v)\leq\sup_{\partial\Om\cup R}((u-v)^+).
\end{equation*}
Noticing that $u-v\leq 0$ on $\partial\Om$ and that $\overline{\intt{\Om_1}}=\overline\Om$, we obtain that 
\begin{equation*}
    \begin{split}
        u(z)-v(z)&\leq\sup_\Om(u-v)
        =\sup_{\Om_1}(u-v)
        \leq\sup_{\partial\Om_1}((u-v)^+)
        =\sup_{R}((u-v)^+)
    \end{split}
\end{equation*}
for any $z\in\Om$. We are left to show that $\sup_{R}((u-v)^+)\leq 0$. Indeed, assume by contradiction that $\sup_{R}((u-v)^+)> 0$. Since $R$ is compact,  there exists $z_0\in R$ such that $$u(z_0)-v(z_0)=\sup_{R}((u-v)^+)=\sup_{R}(u-v).$$ 
Moreover,  $z_0$ is a maximum point for $u-v$ on $\overline\Om$. Let us fix $y_0\in\partial\Om$ such that $\dfz(z_0)=\|z_0-y_0\|$. Then, thanks to Proposition \ref{utile}, it is easy to see that 
\begin{equation}\label{dfac}
    \dfz(z)=\|z-y_0\|
\end{equation}
for any $z$ belonging to $(y_0,z_0)$, the segment connecting $y_0$ and $z_0$. Let now $\nu:=\frac{y_0-z_0}{|y_0-z_0|}$. By \eqref{dfac} it holds that $v(z)<v(z_0)$ for any $z\in (y_0,z_0)$, and moreover
\begin{equation}\label{strict}
    D_\nu^+ v(z_0):=\lim_{h\to 0^+}\frac{v(z_0+h\nu)-v(z_0)}{h}<0.
\end{equation}
Since $z_0$ is a maximum point of $u-v$, it holds in particular that $D_\nu^+ u(z_0)\leq D_\nu^+ v(z_0)$, which implies, together with \eqref{strict}, that $D_\nu^+ u(z_0)=D_\nu u(z_0)<0$. This proves that $Du(z_0)\neq 0$. Since then $z_0$ is a regular point for $u$, the level set $\{z\in\Om\,:\,u(z)=u(z_0)\}$ is locally a $C^2$ hypersurface. Therefore there exists a small Euclidean ball $B$ such that $B$ is tangent to the level set at $z_0$ and moreover $B\subseteq\{z\in\Om\,:\,u(z)\geq u(z_0)\}$. Now, since by our assumptions the Finsler balls relative to $-K_0$ are uniformly convex and $C^2$, there exists $\varrho>0$ and $x_0\in\Om$ such that
\begin{equation}\label{pelota}
    \overline{B_{-K_0}(x_0,\varrho)}\subseteq\{z\in\Om\,:\,u(z)\geq u(z_0)\}
\end{equation}
and  $B_{-K_0}(x_0,\varrho)$ is tangent to $B$ at $z_0$. Indeed, fix a Finsler ball tangent to $B$ at $z_0$ relative to $-K_0$, say $B_F$. On one hand, the principal curvatures of $\partial B$ at $z_0$ are fixed. On the other hand, noticing that the principal curvatures of a $C_+^2$ convex set admit a positive lower bound, we can dilate and translate $B_F$ to make the curvature of $B_F$ as big as we want to ensure that \eqref{pelota} holds. Notice that
\begin{equation}\label{dmax}
    \dfz(z)\geq \dfz(z_0)
\end{equation}
for any $z\in \overline{B_{-K_0}(x_0,\varrho)}$. Indeed, if by contradiction there exists $z\in\overline{B_{-K_0}(x_0,\varrho)}$ such that $\dfz(z)<\dfz(z_0)$, then \eqref{pelota} would imply
\begin{equation*}
    u(z)-k\dfz(z)\geq u(z_0)-k\dfz(z)>u(z_0)-k\dfz(z_0),
\end{equation*}
a contradiction to the maximality of $z_0$. Let now $w_0\in\partial\Om$ be such that $\df(x_0)=\|x_0-w_0\|$, and let $b_0$ be the unique point of intersection between $\partial B_{-K_0}(x_0,\varrho)$ and the segment joining $w_0$ and $x_0$. Then by \eqref{pallastorta}, \eqref{dfac}, \eqref{dmax}, the choice of $b_0$ and the strict convexity of $K_0$, it holds that
\begin{equation*}
    \begin{split}
        \dfz(x_0)=\|x_0-w_0\|
        =\|x_0-b_0\|+\|b_0-w_0\|
        =\varrho+\dfz(b_0)
        \geq \varrho+\dfz(z_0).
    \end{split}
\end{equation*}
On the other hand, \eqref{pallastorta} and the triangle inequality imply
\begin{equation*}
    \begin{split}
        \dfz(x_0)\leq\|x_0-y_0\|
        \leq \|x_0-z_0\|+\|z_0-y_0\|
        =\varrho+\dfz(z_0).
    \end{split}
\end{equation*}
Putting together the previous inequalities we get that
\begin{equation}\label{intheend}
    \dfz(x_0)=\|x_0-y_0\|= \|x_0-z_0\|+\|z_0-y_0\|,
\end{equation}
from which in particular we conclude, exploiting again the strict convexity of $K_0$, that $x_0$ lies on $(y_0,x_0)$. Therefore, thanks to this fact, the first equality in \eqref{intheend} and Proposition \ref{utile}, we conclude that $z_0\in\intt{\Om_1}$, which is a contradiction. In the end we proved that
\begin{equation*}
    \sup_\Om u\leq \sup_{\partial\Om}|\varphi|+k\max_{\overline\Om}\dfz.
\end{equation*}
Since the converse estimate can be obtained in a similar way, the thesis is proved.
\end{proof}
\begin{remark}
    We point out that the proof of Proposition \ref{th:Cinfestimateconst} does not hold for $n\ge2$. Indeed, when $n\geq 2$, the Euclidean mean curvature $H_{\Sigma_{d}}$ in equation \eqref{perdopo2.0}  may blow down to $-\infty$ close to the ridge $R$ even though the Finsler mean curvature $H_{K_0,\Sigma_{d}}$ is strictly positive on $\Om_1$.
\end{remark}
\section{Existence of Lipschitz minimizers for the sub-Finsler functional $\mathcal{I}$}
\label{sc:mainth}
Thanks to the \emph{a priori} estimates of the previous section, together with Proposition \ref{gt} and the uniformity of the estimates with respect to $\eps\in(0,1)$ and $\eta\in (0,\eta_0)$, we are in position
to pass to the limit and find a solution to the sub-Finsler Prescribed Mean Curvature equation.  

\begin{theorem}
\label{th:main}
Let $K_0\in C^{\infty}_+$ be a convex body such that $0  \in \intt K_0$. Let $\Omega \subseteq \rr^{2n}$ be a bounded domain with $ C^{2,1}$ boundary. Let $\varphi \in C^{2,\alpha}(\overline{\Omega})$, for $0<\alpha<1$, and let $F\in C^{1,\alpha}(\overline\Om,\rr^{2n})$  be such that \eqref{taiwan} is satisfied. Assume that $H$ is a constant such that  \eqref{in:hip} and \eqref{curvcond} hold. Then, there exists $\eta_0\in (0,1)$ such that for any $\varepsilon\in(0,1)$ and any $\eta\in (0,\eta_0)$, there exists a function $u_{\varepsilon,\eta}\in C^{2,\alpha}(\overline\Om)$ which solves \eqref{riemreg}. Moreover, there exists a constant $M>0$, independent of $\eps\in (0,1)$ and $\eta\in(0,\eta_0)$, such that any solution $u_{\varepsilon,\eta}$ to \eqref{riemreg} satisfies
\begin{equation}\label{unifbound}
\sup_{\Om} |u_{\varepsilon,\eta}|+\sup_{\Om} |\nabla u_{\varepsilon,\eta}|\le M.
\end{equation}
Finally, there exists a Lipschitz continuous minimizer $u_0 \in \text{Lip}(\overline{\Omega})$ for the functional $\mathcal{I}$ defined in \eqref{eq:Ifun} with $u_0= \varphi$ on $\partial \Omega$.
\end{theorem}

\begin{proof}
 By Proposition \ref{th:Cinfestimate}, Proposition \ref{prop:maxprinc} and Proposition \ref{prop:boundgradest}, there exists a constant $M>0$ such that, for any $\sigma\in[0,1]$,  any $0<\varepsilon<1$  and any $\eta\in(0,\eta_0)$ with $\eta_0>0$ small enough, then any solution $u\in C^{2,\alpha}(\overline\Om)$ to the problem \eqref{sigmadir} satisfies
\[
\sup_{\Om} |u|+\sup_{\Om} |\nabla u|\le M.
\]
Then by Proposition \ref{prop:Fex} there exists a solution $u_{\eps,\eta} \in C^{2,\alpha}(\bar{\Om})$ to
	\begin{equation*}
\begin{cases}\divv(\pi_\eps^h(\nabla u+F))+\eta\divv\left(\frac{\nabla u+\sigma F}{\sqrt{1+|\nabla u+\sigma F|^2}}\right)= H&\text{ in  }\Om\\
			u=\varphi &\text{ in } \partial\Om.
		\end{cases}
	\end{equation*}
Again by Proposition \ref{th:Cinfestimate},  Proposition \ref{prop:maxprinc} and Proposition \ref{prop:boundgradest}, we have that
\begin{equation}
\label{eq:estuneps}
    \sup_{\Om} |u_{\eps,\eta}|+\sup_{\Om} |\nabla u_{\eps,\eta}| \le M,
\end{equation}
where the constant $M>0$ is uniform in $0<\varepsilon<1$ and $\eta\in (0,\eta_0)$. Let $\{\eps_j\}_{j \in \nn}\subseteq (0,1)$ and $\{\eta_j\}_{j \in \nn}\subseteq (0,\eta_0)$ be sequences such that $\eps_j \to 0$ and $\eta_j \to 0$ as $j \to \infty $. Since $M$ is uniform in $\eps$ and $\eta$ by \eqref{eq:estuneps} we gain that $\sup_{\Om} |u_{\eps_j,\eta_j}|\le M$ and that for any $z_1,z_2\in\Omega$ 
\begin{equation}
\label{eq:LIP}
    |u_{\eps_j,\eta_j}(z_1)-u_{\eps_j,\eta_j}(z_2)|\le M|z_1-z_2|.
\end{equation}
Then, by Ascoli-Arzelà theorem there exists $u_0 \in C(\overline{\Om})$ such that $u_{\eps_j,\eta_j}\to u_0$ uniformly in $\overline\Om$. It is clear that $u=\varphi$ on $\partial\Om$. Moreover, taking the limit as $j\to 0$ in  \eqref{eq:LIP}, we gain that
\[
\sup_{z_1 \ne z_2} \frac{|u_{0}(z_1)-u_{0}(z_2)|}{|z_1-z_2|} \le M,
\]
thus $u_0$ is Lipschitz. We claim that $u_0$ is a minimizer for $\mathcal{I}$ defined in \eqref{eq:Ifun}. Indeed, we have that $\|u_{\eps_j,\eta_j}\|_{W^{1,1}(\Om)} \le M|\Om|$, $\|u_{0}\|_{W^{1,1}(\Om)} \le M|\Om|$  and $u_{\eps_j,\eta_j}$ converge to $u_0$ in $L^1(\Om)$. Moreover, the function $(p,(x,y))\to \|p+ F(x,y)\|_*$ is positive, continuous and convex in $p$. 
Therefore, by \cite[Theorem 4.1.2]{MR2492985}, $\mathcal{I}$ is lower semicontinuous with respect to the strong $L^1$-topology, from which we have that
\begin{equation}
\label{eq:lsc}
    \mathcal{I}(u_0) \le \liminf_{j\to \infty} \mathcal{I}(u_{\eps_j,\eta_j}).
\end{equation}
For each $v \in W^{1,1}(\Om)$ such that $v-\varphi \in W^{1,1}_0(\Om)$, it follows that
\begin{equation}
\begin{aligned}
    \mathcal{I}(u_{\eps_j,\eta_j})&=\int_{\Omega} \|\nabla u_{\eps_j,\eta_j} +F\|_{*} \, dz + \int_{\Omega} H u_{\eps_j,\eta_j} \, dz\\
    &\le \int_{\Omega} (\eps_j^3+\|\nabla u_{\eps_j,\eta_j}+F\|_{*}^3)^{\frac{1}{3}} \, dz + \int_{\Omega} H u_{\eps_j,\eta_j} \, dz+\eta_{j}\int_\Om \sqrt{1+|\nabla u_{\eps_j,\eta_j}+F|^2}\,dz\\
    &\le \int_{\Omega} (\eps_j^3+\|\nabla v +F\|_{*}^3)^{\frac{1}{3}} \, dz + \int_{\Omega} H v \, dz+\eta_{j}\int_\Om \sqrt{1+|\nabla v+F|^2}\,dz\\
 &\le  \eps_j |\Om|+\int_{\Omega}\|\nabla v +F\|_{*} \, dz + \int_{\Omega} H v \, dz+\eta_{j}\int_\Om \sqrt{1+|\nabla v+F|^2}\,dz,\\
\end{aligned}
\label{eq:minimires}
\end{equation}
where we have used the fact that the Dirichlet solution $u_{\eps_j,\eta_j} \in C^{2\,\alpha}(\bar{\Om})$ is a minimizer for the functional $ v \to \int_{\Omega} (\eps_j^3+\|\nabla v +F\|_{*}^3)^{\frac{1}{3}} + \int_{\Omega} H v +\eta_{j}\int_\Om \sqrt{1+|\nabla v+F|^2}\,dz$   for each $v \in W^{1,1}(\Om)$ s.t. $v-\varphi \in W^{1,1}_0(\Om)$. Passing to the liminf in \eqref{eq:minimires} and taking into account \eqref{eq:lsc}, we obtain $\mathcal{I}(u_0) \le \mathcal{I}(v)$ for each $v \in W^{1,1}(\Om)$ s.t. $v-\varphi \in W^{1,1}_0(\Om)$.  
\end{proof}

We now apply the same argument of the previous proof in $\hh^1$, using the height estimate provided by Proposition \ref{th:Cinfestimateconst} instead of the one given in Proposition \ref{th:Cinfestimate} to avoid condition \eqref{in:hip}, to obtain the following sharp result in the first Heisenberg group.

\begin{theorem}\label{th:main23}
Let $n=1$ and $K_0\in C^{\infty}_+$ be a convex body such that $0  \in \intt K_0$. Let $\Omega \subseteq \rr^{2}$ be a bounded domain with $ C^{2,1}$ boundary. Let $\varphi \in C^{2,\alpha}(\overline{\Omega})$, for $0<\alpha<1$, and let $F\in C^{1,\alpha}(\overline\Om,\rr^{2})$  be such that \eqref{taiwan} is satisfied. Assume that $H$ is a constant such that \eqref{curvcond}  holds. Then,  there exists $\eta_0\in (0,1)$ such that for any $\varepsilon\in(0,1)$ and any $\eta\in (0,\eta_0)$, there exists a function $u_{\varepsilon,\eta}\in C^{2,\alpha}(\overline\Om)$ which solves \eqref{riemreg}. Moreover, there exists a constant $M>0$, independent of $\eps\in (0,1)$  and $\eta\in(0,\eta_0)$, such that any solution $u_{\varepsilon,\eta}$ to \eqref{riemreg} satisfies
\begin{equation}\label{unifbound2}
\sup_{\Om} |u_{\eps,\eta}|+\sup_{\Om} |\nabla u_{\eps,\eta}|\le M.
\end{equation}
Finally, there exists a Lipschitz continuous minimizer $u_0 \in \text{Lip}(\overline{\Omega})$ for the functional $\mathcal{I}$  defined in \eqref{eq:Ifun} with $u_0= \varphi$ on $\partial \Omega$.
\end{theorem}
To conclude this section, according to \cite{MR2262784} we point out that the Dirichlet problem for the prescribed $K_0$-mean curvature equation can be equivalently stated by means of a weak formulation which takes into account the presence of the singular set. Indeed, given a bounded domain $\Om\subseteq\rr^{2n}$, $\varphi\in W^{1,1}(\Om)$, $H\in L^\infty(\Om)$ and $F\in L^1(\Om)$, we say that $u\in W^{1,1}(\Om)$ is a \emph{weak solution} to the Dirichlet problem for the prescribed $K_0$-mean curvature equation if $u-\varphi\in W^{1,1}_0(\Om)$ and 
\begin{equation}
    \int_{\Om_0}\|\nabla\phi\|_*\,dz+\int_{\Om\setminus \Om_0}\langle\pi(\nabla u+F),\nabla \phi\rangle\,dz+\int_\Om H\phi\,dz\geq 0
\end{equation}
for any $\phi\in W^{1,1}_0(\Om)$, where we recall that $\Om_0=\{\nabla u+F=0\}$. The equivalence between the two formulations is proved in \cite{MR2262784} for the sub-Riemannian setting and can be carried out for the sub-Finsler setting with slight modifications.

\begin{remark}
A deeper look to \cite{MR2305073,MR2094267} suggests that it should be possible to prove that the aforementioned results still hold only assuming that $K_0$ is a convex body in $C^{2,\alpha}_+$ with $0\in\intt{K_0}$, for some $0<\alpha<1$. Accordingly, it is reasonable that in Theorem \ref{th:main} the regularity of $\partial K_0$ can be weakened to $C^{2,\alpha}$, for some $0<\alpha<1$. 
\end{remark}

\section{A sharp existence result of Lipschitz minimizers in the sub-Riemannian setting}\label{sec:last}

As pointed out in the introduction, a Finsler approximation scheme for \eqref{int:pmce} cannot be arbitrarily chosen, since one needs to guarantee classical regularity of the resulting equations. Nevertheless, for a particular class of Finsler metrics, it is possible to choose a more natural approximation scheme. More precisely, let us consider the one-parameter family of   differential equations defined formally by
\begin{equation}\label{riemapprox6}
    \divv\left(\pi_{K_0}(\nabla u+F)\frac{\|\nabla u+F\|_*}{\sqrt{\eps^2+\|\nabla u+F\|^2_*}}\right)=H.
\end{equation}
We point out that, when $K_0$ is the Euclidean unit ball centered at the origin, \eqref{riemapprox6} reduces to the well-known elliptic approximating equation considered for instance in \cite{MR2262784} (cf. Remark \ref{reminfondo}).  In order to give to equation \eqref{riemapprox6} a pointwise meaning, we must impose \emph{a priori} that the function $\tilde G(p):=\|p\|_*\pi_{K_0}(p)$, which is $C^1$ outside the origin, admits a $C^1$ extension to the whole $\rr^{2n}$. This regularity hypothesis turns out to be equivalent to the fact that the left-invariant sub-Finsler structure induced by $K_0$ comes from an underlying left-invariant sub-Riemannian metric on the distribution $\mathcal{H}$ (cf. \cite{MR0208534}), or equivalently that $K_0$ is an ellipsoid centered at $0$. More precisely, it is easy to check that, if $\tilde G\in C^1(\rr^{2n},\rr^{2n})$, then $D\tilde G$ is necessarily a constant, symmetric and positive definite matrix,  and moreover
\begin{equation}\label{formaesplicita}
    \|p\|_*=\sqrt{p\cdot D\tilde G\cdot p^T}\qquad\text{and}\qquad\pi_K(p)=\frac{D\tilde G\cdot p^T}{\|p\|_*}
\end{equation}
for any $p\in\rr^{2n}$. When \eqref{formaesplicita} holds, a direct computation shows that \eqref{riemapprox6} is a well-defined, quasi-linear elliptic equation, so that in this setting a Euclidean regularization term as in \eqref{findiv24} is no longer needed. In order to solve the Dirichlet problem associated to \eqref{riemapprox6} it is then possible to replicate almost word-by-word the computations of Section \ref{sc:aprioriestimate}, with the advantage that the absence of the Euclidean curvature term makes the process easier. 
The main benefit of this new approximation is that, due to the absence of the Euclidean curvature term, a result analogous to Proposition \ref{th:Cinfestimateconst} actually holds for any $n\geq 1$. We include the proof for the sake of completeness. 
\begin{proposition}
\label{th:Cinfestimateconstriem}
Assume that $K_0\in C^{\infty}_+$ induces a left-invariant sub-Riemannian metric on $\hh^n$. Let $\Om\subseteq \rr^{2n}$ be a bounded domain with $C^{2,1}$ boundary, let $\varphi\in C^2(\overline\Om)$ and let $H$ be a constant which satisfies \eqref{curvcond}. There exists a
constant $C_1=C_1(K_0,\Om,\varphi,H,F)>0$, independent of $\sigma\in[0,1]$ and $\eps\in (0,1)$, such that, for any solution  $u\in C^2(\overline\Omega)$ to 
\begin{equation*}\label{sigmadirriem}
		\begin{cases} \divv\left(\pi_{K_0}(\nabla u+\sigma F)\frac{\|\nabla u+\sigma F\|_*}{\sqrt{\eps^2+\|\nabla u+\sigma F\|^2_*}}\right)=\sigma H&\text{ in  }\Om\\
			u=\sigma\varphi &\text{ in } \partial\Om
		\end{cases}
	\end{equation*}
it holds that
\[
\|u\|_{\infty,\Om}\leq C_1.
\]
\end{proposition}
\begin{proof}
Let $C_3=C_3(K_0,\Om,H)$ be as in \eqref{curvcond2}.
Let us define the function $v:\intt{\Om_1}\scu\rr$ as in \eqref{barrierafinale}, that is
\begin{equation*}
    v(z):=\sup_{\partial\Om}|\varphi|+k\dfz(z)
\end{equation*}
for any $z\in\Om_1$, where $k>0$ has to be chosen and $\Om_1$ is the set defined in \eqref{om1}. Again we know (cf. \eqref{diff}) that $v\in C^2(\intt{\Om_1})$.  We repeat again, with minor modifications, the computations of the proof of Proposition \ref{prop:boundgradest} up to \eqref{perdopo}. As in the proof of Proposition \ref{th:Cinfestimateconst}, being $H$ constant, we can choose $C_5=0$ in \eqref{Hnocost}. Moreover, since $\eta=0$, the analogous of \eqref{perdopo} becomes
\begin{equation*}\label{perdopo3.0}
    \begin{split}
        \divv(\pi(\nabla d))(z)+|H|(z_0)
        &=-H_{K_0,\Sigma_{d(z)}}(z)+|H|(z_0)\\
        &\leq-H_{K_0,\partial\Om}(z_0)+|H|(z_0)\\
        &\leq-3C_3<0.
    \end{split}
\end{equation*}
Hence there exists $k>0$, independent of $\eps\in(0,1)$, $\sigma\in[0,1]$ and $z\in\Om_1$, such that $v$ is a subsolution to \eqref{sigmadirriem} on $\intt{\Om_1}$. The thesis then follows \emph{verbatim} as in the proof of Proposition \ref{th:Cinfestimateconst}.
\end{proof}

Therefore, in the sub-Riemannian setting, we can exploit Proposition \ref{th:Cinfestimateconstriem} to avoid condition \eqref{in:hip}, so that the following sharper analogous to Theorem \ref{th:main} holds.
\begin{theorem}
\label{th:main2}
Assume that $K_0\in C^{\infty}_+$ induces a left-invariant sub-Riemannian metric on $\hh^n$.
Let $\Omega \subseteq \rr^{2n}$ be a bounded domain with $ C^{2,1}$ boundary. Let $\varphi \in C^{2,\alpha}(\overline{\Omega})$, for $0<\alpha<1$, and let $F\in C^{1,\alpha}(\overline\Om,\rr^{2n})$  be such that \eqref{taiwan} is satisfied. Assume that $H$ is a constant such that \eqref{curvcond} holds. Then, for any $\varepsilon\in(0,1)$, there exists a function $u_\varepsilon\in C^{2,\alpha}(\overline\Om)$ which solves the Dirichlet problem associated to \eqref{riemapprox6} with boundary datum $\varphi$.
 Moreover, there exists a constant $M>0$, independent of $\eps\in (0,1)$, such that any solution $u_{\eps}$ to \eqref{riemapprox6} satisfies
\begin{equation*}\label{unifbound3}
\sup_{\Om} |u_\eps|+\sup_{\Om} |\nabla u_\eps|\le M.
\end{equation*}
Finally, there exists a Lipschitz continuous minimizer $u_0 \in \text{Lip}(\overline{\Omega})$ for the functional $\mathcal{I}$ defined in \eqref{eq:Ifun} with $u_0= \varphi$ on $\partial \Omega$.
\end{theorem}
\bibliography{references}
\end{document}